\documentclass[10pt]{amsart}

\usepackage{graphicx}              
\usepackage{amsmath,amssymb,amsthm,amsfonts,dsfont,mathrsfs,stmaryrd,float}
\usepackage[latin1]{inputenc}
\usepackage[T1]{fontenc}
\usepackage{enumitem}
\usepackage[style=alphabetic,natbib=true,backend=bibtex]{biblatex}
\usepackage{hyperref}
\usepackage{color}
\usepackage[usenames,dvipsnames]{xcolor}
\usepackage{bm}
\usepackage[cal=boondoxo]{mathalfa}
\usepackage[normalem]{ulem}

\allowdisplaybreaks[4]

\newtheorem{thm}{Theorem}[section]
\newtheorem{lem}[thm]{Lemma}
\newtheorem{prop}[thm]{Proposition}
\newtheorem{cor}[thm]{Corollary}
\newtheorem{defn}[thm]{Definition}

\newtheorem{innercthm}{Theorem}
\newenvironment{cthm}[1]
  {\renewcommand\theinnercthm{#1}\innercthm}
  {\endinnercthm}

\newcommand{\A}{\mathcal{A}}

\newcommand{\F}{\mathfrak{F}}
\newcommand{\T}{\Psi}

\newcommand{\Q}{\mathbb{Q}}
\newcommand{\R}{\mathbb{R}}
\newcommand{\Z}{\mathbb{Z}}
\newcommand{\C}{\mathbb{C}}
\newcommand{\N}{\mathbb{N}}
\renewcommand{\S}{\mathcal{S}}
\renewcommand{\SS}{\mathbb{S}}

\renewcommand{\Re}{\mathop{\mathrm{Re}}}
\renewcommand{\Im}{\mathop{\mathrm{Im}}}

\newcommand{\RR}{\mathfrak{R}}
\newcommand{\f}{\mathcal{z}}
\newcommand{\K}{\mathfrak{K}}
\newcommand{\NN}{\mathcal{N}}
\renewcommand{\r}{\mathcal{r}}
\newcommand{\CC}{\mathcal{C}}

\bibliography{bibliography}

\begin{document}

\title{Wandering intervals in affine extensions of self-similar interval exchange maps: the cubic Arnoux-Yoccoz map}

\author{Milton Cobo}
\address{Departamento de Matem\'atica, Universidade Federal do Esp\'{\i}rito Santo, Av. Fernando Ferrari 514, Goiabeiras, Vit\'oria, Brasil.} \email{milton.e.cobo@gmail.com}

\author{Rodolfo Guti\'errez-Romo}
\address{Departamento de Ingenier\'{\i}a
Matem\'atica and Centro de Modelamiento Ma\-te\-m\'a\-ti\-co, CNRS-UMI 2807, Universidad de Chile, Beauchef 851, Santiago,
Chile.}\email{rgutierrez@dim.uchile.cl}

\author{Alejandro Maass}
\address{Departamento de Ingenier\'{\i}a
Matem\'atica and Centro de Modelamiento Ma\-te\-m\'a\-ti\-co, CNRS-UMI 2807, Universidad de Chile, Beauchef 851, Santiago,
Chile.}\email{amaass@dim.uchile.cl}

\begin{abstract}
	In this article we provide sufficient conditions on a self-similar interval exchange map, whose renormalization matrix has complex eigenvalues of modulus greater than one, for the existence of affine interval exchange maps with wandering intervals and semi-conjugate with it. These conditions are based on the algebraic properties of the complex eigenvalues and the complex fractals built from the natural substitution emerging from self-similarity. We show that the cubic Arnoux-Yoccoz interval exchange map satisfies these conditions. 
\end{abstract}

\maketitle
\markboth{Milton Cobo, Rodolfo Guti\'errez-Romo, Alejandro
Maass}{Wandering intervals in affine extensions of i.e.m.: the cubic Arnoux-Yoccoz map}

\section{Introduction}\label{sec:introduction}

The existence of wandering intervals in dynamical systems has been studied for a long time. \citeauthor{denjoy} \cite{denjoy} proved that an orientation-preserving $\mathcal{C}^1$-diffeomorphism of the circle with irrational rotation number is conjugate with an irrational rotation if and only if it has no wandering intervals, and constructed examples of $\mathcal{C}^r$-diffeomorphisms of this type with wandering intervals for $r < 2$. The absence of wandering intervals is ensured for $\mathcal{C}^2$-diffeomorphisms.
\smallbreak

By suspending a rotation, one obtains a linear flow on a two-dimensional torus. A natural generalization of these flows are linear flows on surfaces of a higher genus, which when restricted to a Poincar\'e section induce interval exchange maps. In this sense, interval exchange maps are natural generalizations of rotations of the circle.
\smallbreak

A bijective map $T \colon [0, 1) \to [0, 1)$ is said to be an \emph{interval exchange map} (i.e.m.)\ if there exists a finite partition $(I_a;a \in \A)$ of $[0, 1)$ made of intervals such that $T(t) = t + \delta_a$ for each $t \in I_a$ and $a \in \A$. Clearly $T$ is a piecewise isometry of the unit interval exchanging the intervals $(I_a;a \in \A)$. An i.e.m.\ $T$ is said to be \emph{self-similar} if there exists $0 < \alpha < 1$ such that the map $T^{(1)}\colon [0, \alpha) \to [0, \alpha)$ of first return by $T$ to the interval $[0, \alpha)$ is up to rescaling equal to $T$. The natural symbolic extension of a self-similar i.e.m.\ is generated by a substitution.
An {\it affine} interval exchange map (affine i.e.m.)\ $f \colon [0, 1) \to [0, 1)$ is a bijective piecewise affine map with positive slopes. The vector $(\ell_a; a\in\A)$, where $\ell_a$ is the slope of $f$ in the $a$-th interval of continuity, is called \emph{the slope vector} of $f$.
\smallbreak

\citeauthor{levitt} \cite{levitt} constructed an example of a non-uniquely ergodic affine i.e.m.\ with wandering intervals, showing that there exist Denjoy counterexamples of arbitrary smoothness
in some surfaces of genus at least $2$.
Given a slope vector and a self-similar i.e.m., \citeauthor{cam-gut} \cite{cam-gut} provided necessary and sufficient conditions for the existence of an affine i.e.m.\ with the same number of intervals and such slope vector which is semi-conjugate with the self-similar i.e.m. Namely, such map exists if and only if the
logarithm of the given slope vector, $\log\ell=(\log\ell_a;a\in\A)$, is orthogonal to the vector of interval lengths $\lambda= (|I_a|;a\in\A)$. 
The self-similarity of the i.e.m.\ implies that $R \lambda = \alpha^{-1}\lambda$, where $R$ is the renormalization matrix (recall that $[0,\alpha)$ is the interval of renormalization) and $\alpha^{-1} > 1$ is the Perron-Frobenius eigenvalue of $R$. Thus, 
basic linear algebra implies that $\log\ell$ is orthogonal to $\lambda$ if and only if $\log\ell$ belongs to the invariant subspace corresponding to all the eigenvalues of $M=R^t$ different from $\alpha^{-1}$. We remark that $M$ is the matrix associated with the substitution associated with the self-similar i.e.m.
\smallbreak

If $\log\ell$ belongs to the stable space of $M$, \citeauthor{cam-gut} \cite{cam-gut} proved that any semi-conjugate affine i.e.m.\ with this slope vector is in fact conjugate with an i.e.m.\ (i.e.\ has no wandering intervals). The resulting conjugacy is of class $\mathcal{C}^{1+\varepsilon}$, with $\varepsilon>0$ depending on the particular eigenspace \cite{barreto}. They 
also built an example of a uniquely ergodic affine i.e.m.\ with wandering intervals that is strictly semi-conjugate with a self-similar i.e.m. This last example and extensions of the results in \cite{cam-gut} were considered more deeply by \citeauthor{cobo} in \cite{cobo}, where a generalization is also obtained by introducing the Rauzy-Veech-Zorich Oseledets decomposition (see \cite{veech-gauss,zorich-gauss}). 
In the case that $\log\ell$ belongs to the unstable space of $M$, \citeauthor{persistence} proved in \cite{persistence} that if it also lies in the eigenspace associated with a real eigenvalue of modulus strictly greater than one, which is different from the Perron-Frobenius eigenvalue $\alpha^{-1}$ but Galois-conjugate with it, then one can choose a semi-conjugate affine i.e.m.\ with such slope vector and wandering intervals. 
If the given vector of logarithms lies in an eigenspace of $M$ associated with the eigenvalue $1$ or $-1$, or if it lies in an invariant subspace corresponding to a conjugate pair of non real eigenvalues of modulus $1$, \citeauthor{bressaud-deviation} in \cite{bressaud-deviation} proved that any semi-conjugate affine i.e.m.\ with such slope vector is indeed conjugate to the i.e.m.\ and, thus, has no wandering intervals.  
Finally, in \cite{yoccoz-wandering}, \citeauthor{yoccoz-wandering} proved that the existence of affine i.e.m.\ with wandering intervals semi-conjugate with a given i.e.m.\ is generic. Thus many non self-similar examples arise. 
\smallbreak

In this article we study the remaining case of the program stated in \cite{cam-gut}.  
That is, we consider a self-similar i.e.m.\ and a slope vector $\ell$ whose logarithm $\log\ell$ 
lies in an invariant subspace of $M$ corresponding to a conjugate pair of non real eigenvalues
of modulus strictly larger than one. We will heavily rely on the strategy of \cite{persistence} and the geometrical models for substitutions defined by \citeauthor{geometricalmodels} in \cite{geometricalmodels}. These geometrical models are often of fractal nature. We will see that some properties of these fractals are sufficient conditions for the existence of an affine i.e.m.\ with wandering intervals that is semi-conjugate with a given self-similar i.e.m.\ having the aforementioned properties. Specifically, we will prove the following theorem:
\smallbreak

\begin{cthm}{A}\label{teo:main0}
Let $T$ be a self-similar i.e.m. Assume that $M$ has an eigenvalue $\beta$ with $|\beta|>1$ such that $\beta/|\beta|$ is not a root of unity, and that there exists an eigenvector $\Gamma$ for $\beta$ such that $T$ has the unique representation property for $\beta$ and $\Gamma$. Then, for almost every $\gamma$ in the complex subspace generated by $\Gamma$, $\exp(-\Re(\gamma))$ can be realized as the slope vector of an affine i.e.m.\ which is semi-conjugate with $T$ and has wandering intervals. 
\end{cthm}

The unique representation property, which will be stated in Section \ref{sec:URP},
is related to the different ways that the extreme points of the dual Rauzy fractal, in the sense of \cite{geometricalmodels}, can be written as certain sums $\sum_{m\geq 1} z_m\beta^{-m}$
with coefficients $z_m$ belonging to a finite subset of the field $\Q[\beta]$.
\smallbreak

Most of the properties we will develop before the proof of the main theorem are of purely symbolic nature, in the sense that they can be proved for a primitive substitution and a non real eigenvalue together with an eigenvector of the matrix associated to such substitution.

The unique representation property needed in this article is in some way analogous the explicit algebraic condition required for the case treated in \cite{persistence}. We think that similar algebraic conditions imply the unique representation property, but we did not succeed to establish a proof, even if interesting examples can be found. However, we found that an algebraic condition similar to the one needed in \cite{persistence} provides a simplification in the hypotheses of Theorem~\ref{teo:main}. Indeed, assume that either $\beta$ is Galois-conjugate with $\alpha^{-1}$, or that $\beta$ is Galois-conjugate with $\alpha$ and the self-similarity comes from Rauzy-Veech renormalizations, that is to say, some iteration of the Rauzy-Veech renormalization returns to the original map. Then $\beta$ is a simple eigenvalue. Moreover, if $ \beta $ is not real, then $ \beta/|\beta| $ is not a root of unity (see Lemma~\ref{lem:simple eigenspaces}). In these cases, since the corresponding eigenspace is one-dimensional, the unique representation property depends only on $ T $ and $ \beta$. Since it is possible to construct an i.e.m.\ that is periodic for the Rauzy-Veech algorithm from most cycles of a Rauzy class, these i.e.m.'s are a natural family of self-similar maps.

Therefore, we have the following consequence of Theorem~\ref{teo:main}:
\smallbreak

\begin{cthm}{B}\label{teo:galoisconjugate}
Let $T$ be an i.e.m.\  and let $\beta$ with $|\beta|>1$ be a non-real eigenvalue of $ M $. Assume that either $\beta$ is Galois-conjugate to $ \alpha $, or that $\beta$ is Galois-conjugate with $ \alpha^{-1}$ and that $T$ is periodic for the Rauzy-Veech renormalization algorithm on the interval $[0,\alpha)$.
If $T$ has the unique representation property for $\beta$, then for almost every associated eigenvector $\gamma$,
$\exp(-\Re(\gamma))$ can be realized as the slope vector of an affine i.e.m.\ which is semi-conjugate with $T$ and has wandering intervals. 
\end{cthm}

\smallbreak

Finally,  we apply the techniques developed along this work to the cubic Arnoux-Yoccoz i.e.m.\  \cite{Arnouxyoccoz}. When defined on the circle, this map is self-similar and its renormalization matrix has an eigenvalue $\beta$ with $|\beta|>1$ such that $\beta/|\beta|$ is not a root of unity. We will discuss this example in Section \ref{sec:example} and prove that:
\smallbreak

\begin{cthm}{C}\label{teo:hyp0}
Let $\beta$ be a non real eigenvalue of the renormalization matrix associated to the cubic Arnoux-Yoccoz i.e.m.\ satisfying $|\beta|>1$ and that $\beta/|\beta|$ is not a root of unity. One has that,
for almost every eigenvector $\gamma$ for $\beta$, there exists a semi-conjugate affine i.e.m.\ with slope vector $\exp(-\Re(\gamma))$ exhibiting wandering intervals.
\end{cthm}

The fact that Theorem \ref{teo:main0} is valid for almost every eigenvector for $\beta$ when considering the cubic Arnoux-Yoccoz i.e.m.\ comes from the simplicity of the eigenvalue $\beta$ in this case.
\medbreak

The article is organized as follows. In Section \ref{sec:strategy} we outline the general strategy introduced in \cite{cam-gut} that will be used to prove the main theorem. In Section \ref{sec:preliminaries} we present some basic properties and definitions concerning self-similar i.e.m.'s. In Section \ref{sec:minimal sequences} the notion of minimal sequence is presented. In Section~\ref{sec:fractals} the fractals associated with $\beta$ and the concept of extreme points are defined. In Section \ref{sec:URP} the unique representation property is introduced.
In Section \ref{sec:main theorem} we restate and prove the main theorems. In Section \ref{sec:example} we prove that the cubic Arnoux-Yoccoz i.e.m.\ satisfies the unique representation property and thus that the main theorem can be applied to it.

\section{Strategy}
\label{sec:strategy}

Let $T \colon [0, 1) \to [0, 1)$ be a self-similar i.e.m.\ with continuity intervals $(I_a; a \in \A)$. Our goal is to prove that under the hypotheses of Theorem \ref{teo:main0} there exists an affine i.e.m.\ $f \colon [0, 1) \to [0, 1)$ with wandering intervals that is semi-conjugate with $T$ or, equivalently, is a topological extension of $T$. 

In order to achieve this we will follow the strategy devised by Camelier and 
Guti\'errez in \cite{cam-gut} and used by Cobo in \cite{cobo} and by Bressaud, Hubert and Maass in \cite{persistence}. That is, we search for a Borel probability measure $\mu$  with atoms on $[0, 1)$ that assigns positive measure to every open interval with the following property: for each $a \in \A$ there exists a positive real 
$\ell_a$ such that 
\begin{equation} \label{eq:measures}
\mu( T(J) ) = \ell_a \mu(J),
\end{equation}
for every Borel set $J \subseteq I_a$. By following the proof of Lemma 3.6 of \cite{cam-gut} and using such measure, one constructs an affine i.e.m.\ that is semi-conjugate with $T$ having a wandering interval. Indeed, assume that $\mu$ satisfies the aforementioned properties. Let $g \colon [0, 1) \to [0, 1)$ be defined by $g(t) = \mu([0, t))$ and $g(0) = 0$. We have that $g$ is strictly increasing, since $\mu$ is positive on open intervals. It is also right-continuous. Therefore, $G = [0, 1) \setminus g([0, 1))$ is a union of countably many intervals of the form $[t_0, t_1)$. Define the map $h \colon [0, 1) \to [0, 1)$ by $h = g^{-1}$ on $[0, 1) \setminus G$ and $h(t) = g^{-1}(t_1)$ if $t \in [t_0, t_1)$, where $[t_0, t_1)$ is a maximal interval in $G$. We have that $h$ is right-continuous, nondecreasing and surjective. Then, define $f \colon [0, 1) \to [0, 1)$ in the following way: if $t \in [0, 1) \setminus G$, we put $f(t) = h^{-1} \circ T \circ h(t)$. If $[t_0, t_1)$ is a maximal interval in $G$, define $f$ linearly and increasing between $[t_0, t_1)$ and $h^{-1}(T \circ h([t_0, t_1)))$. A straightforward computation shows that $h \circ f = T \circ h$ and that $f$ is an affine i.e.m.\ with the desired slope vector.

One way to construct such a measure is finding a sequence $\omega \in \Omega_T$, where $\Omega_T$ is the natural symbolic extension of $T$ constructed from codings of itineraries of points by $T$ with respect to the continuity intervals $(I_a;a \in \A)$, a complex vector $\gamma \in \C^\A$  and a number $\rho > 0$ such that:
\begin{equation}\label{eq:goal0}
\liminf_{n \to \infty} \frac{ \Re( \gamma_n(\omega)) }{n^\rho} > 0 \quad \text{ and } \quad  \liminf_{n \to \infty} \frac{ \Re( \gamma_{-n}(\omega)) }{n^\rho} > 0, 
\end{equation}
where $\gamma_n(\omega) = \gamma_{\omega_0} + \dotsb + \gamma_{\omega_{n-1}}$ and $\gamma_{-n}(\omega) = -(\gamma_{\omega_{-n}} + \dotsb + \gamma_{\omega_{-1}})$ for $n \geq 1$.
This implies that $K = \sum_{n \in \Z} \exp(-\Re(\gamma_n(\omega)) ) < \infty$. Thus, if $\omega\in\Omega_T$ is the coding by $T$ of $t \in [0, 1)$, then the measure
$\mu = \frac{1}{K} \sum_{n \in \Z} \exp(-\Re(\gamma_n(\omega)) ){\mathbf{1}}_{ T^n(t) }$
satisfies the desired properties for the slope vector $\ell=(\exp(-\Re(\gamma_a )); {a \in \A})$. This can be proved in an analogous way to Lemma 22 of \cite{persistence}.

Otherwise, if $\omega$ is not the coding by $T$ of any point in $[0, 1)$, then it is the coding by $T' \colon [0, 1) \to [0, 1)$ of some point in $[0, 1)$, where $T'$ is equal to $T$ up to some discontinuity points. The strategy can be still applied to find $f' \colon [0, 1) \to [0, 1)$ and $h \colon [0, 1) \to [0, 1)$ such that $h \circ f' = T' \circ h$, where $h$ is continuous, surjective, nondecreasing and noninjective. Let $f$ be the right-continuous function that is equal to $f'$ up to a finite number of points. Then $f$ is an affine i.e.m. By right-continuity, $h \circ f = T \circ h$. Moreover, $f$ has wandering intervals since $h(J)$ is a point for some interval $J \subseteq [0, 1)$. Therefore, it is enough to prove \eqref{eq:goal0} for a sequence $\omega$ in $\Omega_T$. This last fact can also be proved using the classical construction of Keane in \cite{keane}.

\section{Background and Preliminaries}\label{sec:preliminaries}

\subsection{Interval exchange maps and affine interval exchange maps}

A bijective map $T \colon [0, 1) \to [0, 1)$ is said to be an \emph{interval exchange map} (i.e.m.)\ if there exists a finite partition $(I_a;a \in \A)$ of $[0, 1)$ made of intervals such that $T(t) = t + \delta_a$ for each $t \in I_a$ and $a \in \A$. Clearly $T$ is a piecewise isometry of the unit interval exchanging the intervals $(I_a;a \in \A)$. 
The vector $\delta = (\delta_a;a \in \A)$ is called the \emph{translation vector} of $T$. 
An i.e.m.\ $T$ is determined by the following combinatorial data: a \emph{length vector} $\lambda = (\lambda_a;a \in \A)$ of positive entries corresponding to the length of each interval $I_a$ and a pair of bijections 
$\pi_0, \pi_1 \colon \{1, \dotsc, |\A|\} \to \A$ encoding the order of the intervals $(I_a;a \in \A)$ before and after the transformation. 

An i.e.m.\ $T$ is said to be \emph{self-similar} if there exists $0 < \alpha < 1$ such that the map $T^{(1)}\colon [0, \alpha) \to [0, \alpha)$ of first return by $T$ to the interval $[0, \alpha)$ is, up to rescaling, equal to $T$. These maps are called \emph{renormalizable} in \cite{cam-gut}. In \cite{cubicAY} they are called \emph{scale-invariant} and the term self-similar is used for maps such that the induced map on some interval is, up to rescaling and rotation, the same map $T$. Here, for convenience, we keep the notation used in \cite{cam-gut} and used implicitly in \cite{persistence}.

For each $a \in \A$ we define the interval 
$I_a^{(1)} = \alpha I_a$ and denote by $R$ the \emph{renormalization matrix} given by $R_{a,b} = |\{ 0 \leq k \leq r_b - 1; T^k(I^{(1)}_b) \subseteq I_a\}|$, where $r_b$ is the first return time of $I^{(1)}_b$ to $[0, \alpha)$.

Given an i.e.m., the \textit{Rauzy-Veech renormalization algorithm} produces a new i.e.m.\ by considering the first-return map on a specific interval and rescaling the domain to $[0, 1)$. A \textit{Rauzy class} consists of the possible combinatorics for the starting and ending position of $(I_a)_{a \in \A}$ that can be reached from a fixed permutation by applying the Rauzy-Veech algorithm (for details, see \cite{viana}). As explained later, cycles on Rauzy classes provide a natural way to construct self-similar i.e.m. Nevertheless, we will not assume that the self-similar i.e.m.\ comes from this construction, unless otherwise stated.

Self-similar i.e.m.'s are always uniquely ergodic and therefore minimal (see \cite{Veech78}). Recall that $T$ is minimal if any point in $[0,1)$ has a dense orbit. For more details on minimal i.e.m.'s see \cite{keane, viana}.

An {\it affine interval exchange map} (affine i.e.m.)\ $f \colon [0, 1) \to [0, 1)$ is a bijective, piecewise affine map with positive slopes. If $(J_a;a \in \A)$ are the continuity intervals of $f$ we say that $\ell = (\ell_a; a \in \A)$ is its \emph{slope vector}, where $\ell_a > 0$ is the slope of $f$ restricted to $J_a$: $f(t)=\ell_a t + \mathcal{d}_a$ for every $t\in J_a$ and for some translation vector $\mathcal{d}=(\mathcal{d}_a;a\in\A)$. Clearly, an i.e.m.\ is an affine i.e.m.\ with slope vector $\ell=(1,\dotsc,1)$. 

We are interested in affine i.e.m.\ extensions of an i.e.m.\ $T$ with wandering intervals. That is, an affine i.e.m.\ $f$ with wandering intervals such that there exists a continuous, surjective, nondecreasing and noninjective map $h \colon [0, 1) \to [0, 1)$ satisfying $h \circ f = T \circ h$. As mentioned in the introduction the existence of such extensions was already studied in \cite{cam-gut,cobo,persistence,bressaud-deviation,yoccoz-wandering}. 

\subsection{Substitution subshifts and prefix-suffix decomposition}

We refer to \cite{Qu} and \cite{fogg} and references therein for the general theory of substitutions.

Let $\A$ be a finite set or \emph{alphabet}. A \emph{word} is a finite string of symbols in $\A$, namely $w=w_0\ldots w_{m-1}$, where $|w|=m$ is called the \emph{length} of $w$. The \emph{empty word} $\varepsilon$ is defined as the word of length zero. The set of all words in $\A$ is denoted by $\A^*$ and the set of words of positive length is denoted by $\A^+=\A^*\setminus \{\varepsilon\}$.

We will need to consider words indexed by integers. We will write such a word as $w=w_{-m} \ldots w_{-1}\, \cdot \, w_0\ldots w_{n}$, where $m,n$ are nonnegative integers and the dot separates negative and nonnegative coordinates. The set of one-sided infinite sequences $\omega=(\omega_m)_{m \geq 0}$ in $\A$ is denoted by $\A^\N$. Analogously, $\A^\Z$ denotes the set of two-sided infinite sequences $\omega=(\omega_m)_{m\in \Z}$.

A \emph{substitution} is a map  $\sigma \colon \A \to \A^+$. 
It naturally extends to $\A^+$, $\A^\N$ and $\A^\Z$ by concatenation. For
$\omega=(\omega_m)_{m\in \Z} \in \A^{\Z }$ the extension is given by 
$$
\sigma(\omega)=\ldots{} \sigma(\omega_{-2})\sigma(\omega_{-1}) \, \cdot \, \sigma(\omega_0)\sigma(\omega_1) \ldots{} , 
$$
where the central dot separates the negative and nonnegative coordinates
of $\sigma(\omega)$. A further natural convention is that $\sigma(\varepsilon) = \varepsilon$.

Let $M^\sigma$ be the matrix with indexes in $\A$ such that $M^\sigma_{a,b}$ is the number of times the letter $b$ appears in $\sigma(a)$ for any $a,b \in \A$. The substitution is said to be \emph{primitive} if there exists an integer $n \geq 1$ such that, for any $a \in \A$, $\sigma^n(a)$ contains every letter of $\A$, where $\sigma^n$ denotes $n$ consecutive iterations of $\sigma$.

Let $\Omega_\sigma \subseteq \A^\Z$ be the subshift defined from $\sigma$. That is, $\omega \in \Omega_\sigma$ if and only if any subword of $\omega$ is a subword of $\sigma^n(a)$ for some integer $n \geq 0$ and $a \in \A$. We call $\Omega_\sigma$ the \emph{substitution subshift} associated with $\sigma$. This subshift is minimal whenever $\sigma$ is primitive. 

Assume $\sigma$ is primitive. By the recognizability property (see \cite{mosse}), given a point $\omega \in \Omega_\sigma$ there exists a unique sequence $(p_m,c_m,s_m)_{m\geq 0} \in (\A^* \times \A \times \A^*)^\N$ such that for each integer $m \geq 0$ we have $\sigma(c_{m+1})=p_m c_m s_m$ and
$$
	{}\ldots \sigma^3(p_3) \sigma^2(p_2)\sigma^1(p_1) p_0 \, \cdot \, c_0 s_0 \sigma^1(s_1)\sigma^2(s_2)\sigma^3(s_3)\ldots{}
$$
is the central part of $\omega$, where the dot separates negative and nonnegative coordinates. We remark that 
the $p_m$'s (resp.\ $s_m$'s) are in the finite subset of $\A^*$ containing all prefixes (resp.\ suffixes) 
of $\sigma(a)$ for all $a$ in $\A$.
This sequence is called the {\it prefix-suffix decomposition} of $\omega$ (for more details see for instance \cite{CS}). 

If only finitely many suffixes $s_m$'s are nonempty, then there exist $a \in \A$ and nonnegative integers $n$ and $q$ such that
$$
	\omega_{[0,\infty)}=c_0 s_0 \sigma^1(s_1)\ldots \sigma^n(s_n)  \lim_{m\to \infty} \sigma^{m q}(a).
$$
Analogously, if only finitely many $p_m$'s are nonempty, then there exist $a \in \A$ and nonnegative integers $n$ and $q$ such that
$$
	\omega_{(-\infty,-1]}=\lim_{m\to \infty} \sigma^{m q}(a) \sigma^n(p_n) \ldots \sigma^1(p_1) p_0.
$$

The recognizability property also implies that $\Omega_\sigma=\bigcup_{m=0}^{n} S^{-m} (\sigma(\Omega_\sigma))$ for some positive integer $n$, where 
$S \colon\A^\Z\to \A^\Z$ is the left shift map. 

\subsection{Symbolic coding of self-similar i.e.m.'s}
Let $T$ be a self-similar i.e.m.\ and $(I_a;a\in \A)$ its associated intervals. Recall that under the self-similar condition $T$ is minimal. Given $t \in [0,1)$ we construct a symbolic sequence 
$\omega=(\omega_m)_{m\in \Z} \in \A^\Z$, where 
$\omega_m=a$ if and only if $T^m(t) \in I_a$. The sequence $\omega$ is called the \emph{itinerary} of $t$.
Let $\Omega_T \subseteq A^\Z$ be the closure of the set of sequences constructed in this way for every $t \in [0,1)$.  
Clearly the sequence associated with $T(t)$ corresponds to $S(\omega)$, where 
$S \colon\A^\Z\to \A^\Z$ is the left shift map. Moreover, it is classical that there exists 
a continuous and surjective map $\pi_T \colon \Omega_T \to [0,1)$ such that $T\circ \pi_T=\pi_T\circ S$. The map $\pi_T$ is invertible up to a countable set of points corresponding to the orbits of discontinuities of $T$. 
Since $T$ is self-similar, the restriction of $S$ to $\Omega_T$ is minimal and $\Omega_T$ is a substitutive subshift associated with a substitution $\sigma \colon\A\to \A^+$. That is, $\Omega_T = \Omega_\sigma$ for some substitution $\sigma$.
The substitution is constructed in the following way: $\sigma(a)=w_0\ldots w_{r_a-1}$ if and only if $T^m(I^{(1)}_a) \subseteq I_{w_m}$ for all $0\le m \leq r_a - 1$ and $a \in \A$. We then have that $M^\sigma=R^t$, the transpose of the renormalization matrix associated with $T$. For details see \cite{cam-gut}.
	
\section{Minimal sequences associated with a self-similar i.e.m.}\label{sec:minimal sequences}

We fix a self-similar interval exchange map $T$ which is self-induced on the interval $[0,\alpha)$ with $0<\alpha<1$. Let $M=R^t$ be the matrix of its associated substitution $\sigma$ and
$\beta \in \C$ be an eigenvalue of $M$ with $|\beta| > 1$ and such that $\beta/|\beta|$ is not a root of unity. We fix an eigenvector $\gamma$ for $\beta$ for the rest of the section.

\begin{defn}\label{def:sumas_gamma}
For $w=w_0\ldots w_{n-1}  \in \A^*$ we set $\gamma(w) =\gamma_{w_0}+\dotsb+\gamma_{w_{n-1}}$. 
\end{defn}
It is easy to see that for any integer $n\geq 0$, 
\begin{equation}\label{eq:gamma(sigma^n)} 
\gamma(\sigma^n(w)) = \beta^n \gamma(w).
\end{equation}

For a sequence $\omega=(\omega_m)_{m\in\Z} \in \Omega_T$ we define 
$\gamma_0(\omega) = 0$, $\gamma_n(\omega) = \gamma( \omega_0 \ldots \omega_{n - 1} )$ for $n \geq 1$ and $\gamma_n(\omega) = -\gamma(\omega_n \ldots \omega_{-1})$ for $n \leq -1$.

\begin{defn}\label{def:minimal point}
A sequence $\omega \in \Omega_T$ is a \emph{minimal sequence} for the vector 
$\gamma$ if $$\Re( \gamma_n(\omega))\geq 0 \text{ for all } n\in \Z.$$
\end{defn}

Our purpose is to prove that minimal sequences for some eigenvector $\gamma$ satisfy equation \eqref{eq:goal0}, that is, that there exists a number $\rho > 0$ such that:
\begin{equation*}
\liminf_{n \to \infty} \frac{ \Re( \gamma_n(\omega)) }{n^\rho} > 0 \quad \text{ and } \quad  \liminf_{n \to \infty} \frac{ \Re( \gamma_{-n}(\omega)) }{n^\rho} > 0.
\end{equation*}

This does not necessarily hold for an arbitrary eigenvector $\gamma$. 

The next lemma illustrates a property of minimal sequences for finite words. 

\medskip

\begin{lem}
\label{lem:minimal-leq} 
Let $\omega \in \Omega_T$ be a minimal sequence for $\gamma$. Then, for all integers $n \leq -1$ and $m \geq n$ we have that
$$\Re(\gamma(\omega_n \ldots \omega_{-1})) \leq \Re( \gamma (\omega_n \ldots \omega_m) ).$$
\end{lem}

\begin{proof}
If $m \leq -1$, then
$$\Re( \gamma (\omega_n \ldots \omega_{-1}) ) - \Re( \gamma (\omega_n \ldots \omega_m) ) = -\Re( \gamma_{m+1}(\omega)) \leq 0$$
by minimality of $\omega$. If $m \geq 0$, then
$$\Re( \gamma (\omega_n \ldots \omega_m) ) - \Re( \gamma (\omega_n \ldots \omega_{-1}) ) = \Re( \gamma_m (\omega) ) \geq 0$$
by minimality of $\omega$.
\end{proof}

\begin{lem}\label{lem:existence of minimal points}
There exist minimal sequences for $\gamma$.
\end{lem} 
\begin{proof}

We have that $\Re(\gamma)\neq 0$. Indeed, if $\Re(\gamma) = 0$ then $\gamma = i \chi$, where $\chi \in \R^\A$ is an eigenvector associated to $\beta$, i.e., $M \chi = \beta \chi$. Since $M$ is an integer matrix, $M \chi \in \R^\A$. This contradicts the fact that $\beta$ is not real. 

Recall that $\Re(\gamma)$ is orthogonal to $\lambda$. Then, there exist letters $a,b\in \A$ with $\Re(\gamma_a) < 0$ and $\Re(\gamma_b) > 0$. 
Also, by minimality of $T$, there exists an itinerary of the form $a w b$ with $w\in\A^*$.  

Let $\beta_0=\beta/|\beta|$ and consider $(n_k)_{k \geq 1}$ a sequence of integers such that $\Re(\beta_0^{n_k}) > C$ for every integer $k$ and some constant $C > 0$.
In this way there is a decomposition 
$$\sigma^{n_k}(awb)=\sigma^{n_k}(a)\sigma^{n_k}(w) \sigma^{n_k}(b)= \omega_{-m_k}\omega_{-m_k+1}\ldots \omega_{-1} \, \cdot \, \omega_0\omega_1\ldots \omega_{m_k'},$$ with the dot marking the minimal sum with respect to $\gamma$, i.e., 
$$\Re(\gamma(\omega_{-m_k}\ldots \omega_{-1})) \leq \Re(\gamma(\omega_{-m_k}\ldots \omega_m))$$
for every $-m_k \leq m \leq m_k'$. We claim that
both sequences $(m_k)_{k\geq 0}$ and $(m_k')_{k\geq 0}$ go to infinity when $k\to \infty$.
For this we only need to prove that both sequences $\Re(\gamma(\omega_{-m_k}\dots\omega_{-1}))$ and
$\Re(\gamma(\omega_{0}\dots\omega_{m_k^\prime}))$ are unbounded. 
First recall that, from \eqref{eq:gamma(sigma^n)}, $\gamma(\sigma^{n_k}(a)) = \beta^{n_k}\gamma_a = |\beta|^{n_k}\beta_0^{n_k}\gamma_a$ and $\gamma(\sigma^{n_k}(b)) = |\beta|^{n_k}\beta_0^{n_k}\gamma_b$. 
Since the sequence $(\Re(\beta_0^{n_k}))_{k \geq 1}$ is positive and bounded from below, the choice of $a$ and $b$ implies that
$$
\Re(\gamma(\sigma^{n_k}(a))) \to -\infty \quad \text{ and } \Re(\gamma(\sigma^{n_k}(b))) \to \infty.
$$
But from the definition of the decomposition we have that
$$\Re(\gamma(\sigma^{n_k} (a))) \ge \Re(\gamma(\omega_{m_k}\dots\omega_{-1}))
\text{ and } 
\Re(\gamma(\sigma^{n_k} (b))) \le \Re(\gamma(\omega_{0}\dots\omega_{m_k^\prime})).$$
Thus, the claim follows. 

By taking a subsequence one can assume that $\omega_{-m_k}\omega_{-m_k+1}\ldots \omega_{-1} \, \cdot \, \omega_0\omega_1\ldots \omega_{m_k'}$ converges to a 
sequence $\omega \in \A^\Z$. Since $a w b$ is a subword of a point in $\Omega_T$ and $\sigma(\Omega_T)\subseteq \Omega_T$, we obtain that $\omega \in \Omega_T$. Moreover, by construction we have that 
$\Re(\gamma_n(\omega)) \geq 0$ for every $n \in \Z$, which proves that $\omega \in \Omega_T$ is a minimal sequence for $\gamma$. 
\end{proof}

\section{Fractals associated with a self-similar i.e.m.}\label{sec:fractals}

We continue with notations of previous section. So $\beta \in \C$ is an eigenvalue of $M$ with $|\beta| > 1$ such that $\beta_0 = \beta/|\beta|$ is not a root of unity. Along this section $\gamma$ is a fixed eigenvector of $M$ for $\beta$. 

For $a \in \A$ set
$$
\S_a = \{(p_m, c_m, s_m)_{m \geq 1};\sigma(a) = p_1 c_1 s_1 \text{   and   } \sigma(c_{m}) = p_{m+1} c_{m+1} s_{m+1} \text{ for all } m \geq 1 \}.
$$

For $x \in \S_a$ write $x = (p_m^x, c_m^x, s_m^x)_{m \geq 1}$ and define 
functions $\f_a, \f_a^{(n)}\colon \S_a \mapsto \C$ for each $a\in\A$ and integer $n \ge 1$ by
$$\f_a(x) = \sum_{m \geq 1} \beta^{-m} \gamma(p_m^x)\textrm{ and }\f^{(n)}_a(x) = \sum_{m = 1}^n \beta^{-m} \gamma(p_m^x).$$

Clearly, the previous functions and several notations below depend on $\gamma$. To simplify we will not include $\gamma$ in the notations unless it becomes necessary. 

\begin{defn}\label{def:fractal}
The \emph{fractal associated with $a \in \A$} is the set 
$\F_a = \{ \f_a(x); x \in \S_a \}$. We also define $\F^{(n)}_{a} = \{ \f^{(n)}_a(x); x \in \S_a \}$. We say that $x \in \S_a$ is a \emph{representation of $z \in \F_a$} if $z = \f_a(x)$.
\end{defn}

Let $a \in \A$. From the definition of $\S_a$ it is easy to see that for any $x \in \S_a$ we have that $\sigma^n(a)=\sigma^{n-1}(p_1^x)\ldots p_n^x c_n^x s_n^x\ldots \sigma^{n-1}(s_1^x)$ for any integer $n \geq 1$. Furthermore, we have that $\F_a = \overline{ \bigcup_{n \geq 1} \F^{(n)}_{a} }$. Indeed, let $x \in \S_a$ and $n \geq 1$. We can ``truncate'' $x$ in the following way: let $x|_n$ be the sequence defined by $(p_m^{x|_n}, c_m^{x|_n}, s_m^{x|_n}) = (p_m^x, c_m^x, s_m^x)$ for every $1 \leq m \leq n$. For each $m \geq n + 1$ we proceed inductively by defining $(p_m^{x|_n}, c_m^{x|_n}, s_m^{x|_n}) = (\varepsilon, c_m^{x|_n}, s_m^{x|_n})$ in a way such that $\sigma(c_{m}^{x|_n}) = c_{m+1}^{x|_n}s_{m+1}^{x|_n}$. It is easy to see that $x|_n \in \S_a$ and that $\f_a(x|_n) = \f_a^{(n)}(x)$, so $\F^{(n)}_{a} \subseteq \F_a$ and the claim follows. In a similar way, we can construct a point $x \in \S_a$ such that $p_m^x = \varepsilon$ for every $m \geq 1$. Clearly, $\f_a(x) = 0$, so $0 \in \F_a$. 
Moreover, $\F_a \neq \{0\}$. In fact, by primitivity of $\sigma$ and the fact that $\gamma \neq 0$ it is easy to see that, for some integer $n \geq 1$, there exists a prefix $w$ of $\sigma^n(a)$ that satisfies $\gamma(w) \neq 0$. We can then choose $x \in \S_a$ such that $w = \sigma^{n-1}(p_1^x)\ldots p_n^x$. Therefore,
\begin{equation}\label{eq:prefixes of sigma^n(a)}
\begin{split}
	\f_a^{(n)}(x)
	&= \sum_{m = 1}^n \beta^{-m} \gamma(p_m^x) = \frac{1}{\beta^n}\sum_{m = 1}^n \beta^{n-m} \gamma(p_m^x) = \frac{1}{\beta^n}\sum_{m = 1}^n \gamma(\sigma^{n-m}(p_m^x)) \\
	&= \frac{1}{\beta^n} \gamma(\sigma^{n-1}(p_1^x)\ldots p_{n-1}^x) = \frac{1}{\beta^n} \gamma(w) \neq 0.
\end{split}\end{equation}

Finally,  consider $x \in \S_a$ such that $s_m^x = \varepsilon$ for each integer $m \geq 1$, which can be defined inductively as before. We obtain that $\sigma^n(a) = \sigma^{n-1}(p_1^x)\ldots p_n^x c_n^x$ for each integer $n \geq 1$, so a similar computation as \eqref{eq:prefixes of sigma^n(a)} shows that $\f_a(x) = \gamma_a \in \F_a$.

An elementary computation yields that $\F_a$ is a particular case of a fractal built by projecting stepped lines to an expanding plane as defined in \cite{geometricalmodels} and it is not necessarily a Rauzy fractal. Thus, the fractal $\F_a$ is a compact and path-connected subset of the plane. In addition, in \cite{geometricalmodels} 
the authors study the fractals associated with the cubic Arnoux-Yoccoz i.e.m., proving that they have nonempty interior and that they are somehow related to the so-called tribonnacci fractal. 
In Section \ref{sec:example} we will find a parametrization of the boundary of the fractal associated with the cubic Arnoux-Yoccoz map. This map will serve as an example to illustrate our main result and the techniques developed along the article. 
Although this i.e.m.\ is not self-similar in the sense we are following here 
(see  Section \ref{sec:preliminaries}), Theorem \ref{teo:main0} can be still applied by making some slight modifications. 
\smallbreak

We are interested in the extreme points of the defined fractals along directions.

\begin{figure}
	\centering
	\includegraphics[width=0.35\textwidth]{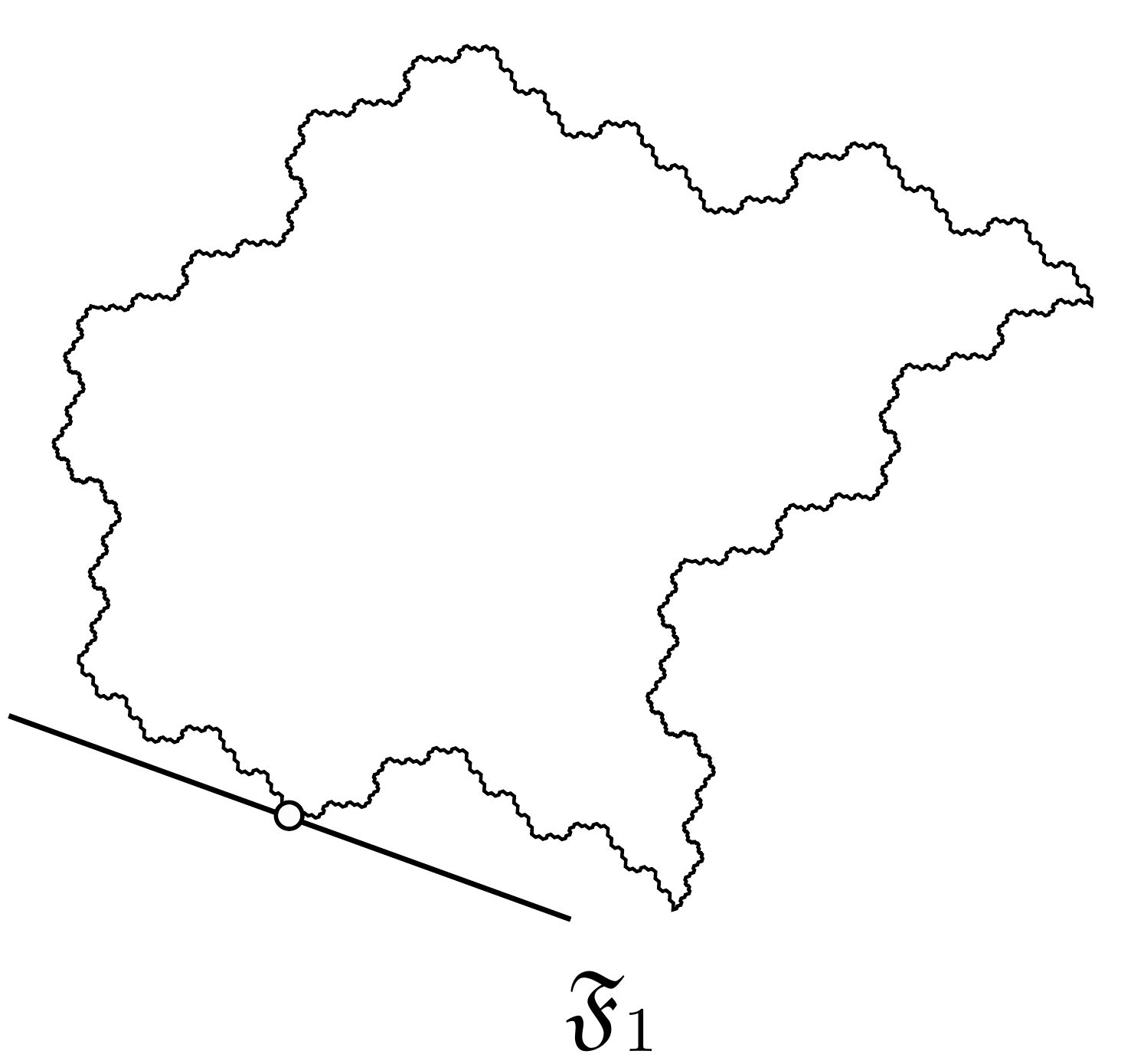} 
	\caption{The dot indicates an extreme point of one Arnoux-Yoccoz fractal in the direction of the line.}
	\label{fig:extreme}
\end{figure}

\begin{defn}\label{def:extremepoints}
For $\tau \in \SS^1$ and an integer $n\geq 1$ define 
$$v_a(\tau) = \min_{z \in \F_a} \Re(\tau z) \text{ and } 
v^{(n)}_{a}(\tau) = \min_{z \in \F^{(n)}_{a}} \Re(\tau z).$$
We call $E_a(\tau) = \{ z \in \F_a; \Re( \tau z ) = v_a(\tau) \}$ the \emph{set of  extreme points} of $\F_a$ for the direction $\tau$ as shown by Figure \ref{fig:extreme}. Clearly $v_a(\tau) \leq v^{(n)}_{a}(\tau) \leq 0$ since $\F_a^{(n)} \subseteq \F_a$ and $0 \in \F_a^{(n)}$. 
\end{defn}

\begin{lem}
\label{lem:continuity of v_a}
The function $v_a \colon \SS^1 \to \R$ is continuous for every $a \in \A$.
\end{lem}
\begin{proof}
Let $(\tau_n)_{n \geq 1}$ be a sequence in $\SS^1$ that converges to $\tau$. Assume first by contradiction that $u = \liminf_{n \to \infty} v_a(\tau_n) < v_a(\tau)$. Let $(n_k)_{k \geq 1}$ be a sequence such that $\lim_{k \to \infty} v_a(\tau_{n_k}) = u$. Let $z_{n_k} \in \F_a$ such that $\Re(\tau_{n_k} z_{n_k}) = v_a(\tau_{n_k})$. Since $\F_a$ is compact, we can assume that $\lim_{k \to \infty} z_{n_k} = z \in \F_a$. We have that $\lim_{k \to \infty} \Re(\tau_{n_k} z_{n_k}) = \Re(\tau z)$, so $\Re(\tau z) = u < v_a(\tau)$, a contradiction. Now, let $z \in \F_a$ such that $\Re(\tau z) = v_a(\tau)$. Clearly $v_a(\tau_n) \leq \Re(\tau_n z)$ and therefore $\limsup_{n \to \infty} v_a(\tau_n) \leq v_a(\tau)$. We obtain that
$$\limsup_{n \to \infty} v_a(\tau_n) \leq v_a(\tau) \leq \liminf_{n \to \infty} v_a(\tau_n),$$
which shows that $\lim_{n \to \infty} v_a(\tau_n) = v_a(\tau)$.
\end{proof}

\subsection{Basic properties of extreme points} 

We will now present some important properties of extreme points that will be used to prove the main theorem. We start with a technical definition.

\begin{defn}
Let $w \in \A^*$ be an itinerary by $T$. A prefix $w'$ of  $w$ is said to be a \emph{minimal prefix} for $w$ and $\gamma$ if 
$$
\Re( \gamma(w') ) = \min\{ \Re( \gamma(w'')) ; w'' \text{ is a prefix of  } w\}.
$$
\end{defn}

\begin{lem}\label{lem:minimo n} 
	Let $n\ge 1$ be an integer, $a\in\A$ and $\tau\in\SS^1$. Let $w$ be a minimal prefix for $\sigma^n(a)$ and the vector $\tau\gamma$. Let $x \in \S_a$ be any sequence that satisfies $w= \sigma^{n-1}(p_1^x)\dots \sigma(p_{n-1}^x)p_{n}^x$. Then $v_a^{(n)}(\beta_0^n\tau) = \Re(\beta_0^n\tau\f_a^{(n)}(x))$.
\end{lem}
 
\begin{proof}
	Take $y = (\bar p_m, \bar c_m, \bar s_m)_{m\ge 1}\in\S_a$ and let $w' = \sigma^{n-1}(\bar p_1)\dots \sigma(\bar p_{n-1}) \bar p_{n}$. By hypothesis,  
	$\Re(\tau\gamma(w)) \le \Re(\tau\gamma(w'))$. By using \eqref{eq:gamma(sigma^n)} we obtain
	that
$$
\gamma(w) = \sum_{m = 1}^n \beta^{n - m} \gamma(p_m^x) = \beta^n \sum_{m = 1}^n \beta^{-m} \gamma(p_m^x) = \beta^n \f_a^{(n)}(x).
$$
	Therefore, multiplying by $\tau$, taking real part and multiplying by $|\beta|^{-n}$ we get:
\begin{align*}
\Re( \beta_0^n \tau  \f_a^{(n)}(x))
& =|\beta|^{-n} \Re( \tau \gamma(w)) \\ 
&\leq |\beta|^{-n} \Re(\tau  \gamma(w')) = \Re( \beta_0^n \tau \f_a^{(n)}(y)),
\end{align*}
where in the last equality we have used that the computation for $w'$ is analogous to the one developed for $w$ just above. Since $y$ is arbitrary, the result follows.
\end{proof} 

\begin{lem}[Continuation property]
\label{lem:continuation}
Let $a \in \A$ and $\tau \in \SS^1$. If $x \in \S_a$ satisfies $\f_a(x) \in E_a(\tau)$, then $\f_{c_1^x}(S(x)) \in E_{c_1^x}(\beta_0^{-1}\tau)$, where $S$ is the left shift map. Moreover,
$$\Re(\tau \f_a(x)) = \Re(\tau \beta^{-1} \gamma(p_1^x)) + |\beta|^{-1} v_{c_1^x}(\beta_0^{-1}\tau).$$
Similarly, if $x \in \S_a$ satisfies $\f_a(x) \in E_a^{(n)}(\tau)$ for an integer $n \geq 1$, then we have that $\f_{c_1^x}(S(x)) \in E_{c_1^x}^{(n-1)}(\beta_0^{-1}\tau)$ and
$$\Re(\tau \f_a^{(n)}(x)) = \Re(\tau \beta^{-1} \gamma(p_1^x)) + |\beta|^{-1} v_{c_1^x}^{(n-1)}(\beta_0^{-1}\tau).$$
\end{lem}
\begin{proof} Let $x \in \S_a$ as in the hypothesis and assume $\f_{c_1^x}(S(x)) \notin E_{c_1^x}(\beta_0^{-1}\tau)$. Then there exists $y \in \S_{c_1^x}$ such that $\Re(\beta_0^{-1}\tau \f_{c_1^x}(y)) < \Re(\beta_0^{-1}\tau \f_{c_1^x}(S(x)))$. Clearly $(p_1^x, c_1^x, s_1^x) y \in \S_a$. Therefore, from the identities  
$\f_a(x)= \beta^{-1} \gamma(p_1^x) + \beta^{-1} \f_{c_1^x}(S(x))$ and  
$\f_a((p_1^x, c_1^x, s_1^x) y))=\beta^{-1} \gamma(p_1^x) + \beta^{-1} \f_{c_1^x}(y)$, we deduce that
$$\Re(\tau\f_a((p_1^x, c_1^x, s_1^x) y)) < \Re(\tau\f_a(x)),$$
which contradicts the fact that $x$ is an extreme point in $E_a(\tau)$. 

The equality
$$\Re(\tau \f_a(x)) = \Re(\tau \beta^{-1} \gamma(p_1^x)) + |\beta|^{-1} v_{c_1^x}(\beta_0^{-1}\tau)$$
follows directly from the fact that $\f_{c_1^x}(S(x)) \in E_{c_1^x}(\beta_0^{-1}\tau)$. The rest of the proof follows analogously.
 
\end{proof}

\begin{lem}[Exponential approximation]
\label{lem:minimum-exponential}
Let $a \in \A$ and $\tau \in \SS^1$. For all $n \geq 1$ we have 
$|v_{a}^{(n)}(\tau) - v_a(\tau)| \leq C |\beta|^{-n}$, where
$C = \max\{ -v_b(\tau); b \in \A, \tau \in \SS^1 \} < \infty$.
\end{lem}

\begin{proof}
Let $x \in \S_a$ with $\f_a(x) \in E_a(\tau)$. Then, 
$$v_a(\tau)=\Re(\tau \f_a(x)) \leq v_a^{(n)}(\tau) \leq \Re(\tau \f_a^{(n)}(x)).$$
Therefore,
\begin{align*}
|v_a^{(n)}(\tau) - v_a(\tau) | = v_a^{(n)}(\tau) - v_a(\tau)
&\leq -\Re\left( \tau \sum_{m \geq n + 1} \beta^{-m} \gamma(p_m^x) \right) \\ 
&= -|\beta|^{-n} v_{c_n^x}(\beta_0^{-n}\tau),
\end{align*}

where in the last equality we used the Continuation Property of Lemma \ref{lem:continuation}.
\end{proof}

\begin{lem}\label{lem:minimal-not-periodic}
For any $a \in \A$, eventually periodic elements of $\S_a$ do not produce extreme points in 
$\F_a$.
\end{lem}
\begin{proof}
Let $x$ be an eventually periodic element of $\S_a$. Then, there exist nonnegative integers $n_0$ and $q$ such that 
for every $k \geq 0$ and $1 \leq m \leq q$ we have
$$
(p_{n_0 + kq + m}^x, c_{n_0 + kq + m}^x, s_{n_0 + kq + m}^x)=(p_m, c_m, s_m).
$$
Assume that $\f_a(x)$ is an extreme point for the direction $\tau \in \SS^1$. 
By definition and periodicity, for every $k \geq 1$, 

\begin{align*}
&\Re(\tau \f_a(x)) \\
{}={}& \Re\left( \tau \sum_{m = 1}^{n_0 + kq} \beta^{-m} \gamma(p_m^x) \right)
+ |\beta|^{-n_0-kq} \Re\left( \beta_0^{-n_0-kq} \tau  \sum_{j \geq 1}\beta^{-j} \gamma(p^x_{n_0+kq+j}) \right) \\
{}={}& \Re\left( \tau \sum_{m = 1}^{n_0 + kq} \beta^{-m} \gamma(p_m^x) \right)
+ |\beta|^{-n_0-kq} \Re\left( \beta_0^{-n_0-kq} \tau  \sum_{j \geq 0}\beta^{-jq} \sum_{m = 1}^{q} \beta^{-m} \gamma(p_m) \right).
\end{align*}
Define $z = \sum_{j \geq 0}\beta^{-jq} \sum_{m = 1}^{q} \beta^{-m} \gamma(p_m)$, then,
$$
\Re(\tau \f_a(x)) = \Re\left( \tau \sum_{m = 1}^{n_0 + kq} \beta^{-m} \gamma(p_m^x) \right) + |\beta|^{-n_0-kq} \Re(\beta_0^{-n_0-kq} \tau z).
$$
By Lemma \ref{lem:continuation}, 
$\Re( \beta_0^{-n_0-kq}\tau z)=v_{c^x_{n_0+kq}}(\beta_0^{-n_0-kq}\tau)=v_{c_{q}}(\beta_0^{-n_0-kq}\tau)$.  Thus $\Re(\beta_0^{-n_0-kq}\tau z) \leq 0$ for any $k\geq 0$. Since $(\beta_0^{-kq})_{k\geq 0}$ is dense in $\SS^1$, we deduce that $z =0$ and consequently $v_{c_{q}}(\beta_0^{-n_0-kq}\tau)=0$ for every $k\geq 0$. Finally, by continuity of $v_a$ (Lemma \ref{lem:continuity of v_a}) and density, $v_{c_{q}}(\xi)=0$ for every $\xi \in \SS^1$. Since $\F_{c_{q}}$ is not reduced to $\{0\}$ we get a contradiction.  
\end{proof}	
We remark that from previous lemma one deduces that $v_a(\tau) < 0$ for any $a \in \A$ and $\tau \in \SS^1$. Indeed, we know that  
$v_a(\tau)\leq 0$, but if $v_a(\tau) = 0$ then we would have that $0 \in \F_a$ is an extreme point, which contradicts the previous lemma.

\subsection{The set \texorpdfstring{$\bm{\Psi}$}{Psi} and extreme points}

\begin{defn}\label{def:a,(pcs)}
Let $a\in\A$ and $(p,c,s)\in\A^*\times\A\times\A^*$ such that $\sigma(a)=pcs $. We denote by $\S_{a, (p,c,s)}$ the set of sequences $x\in \S_a$ starting with $x_1=(p, c, s)$. We define
$$\F_{a,(p,c,s)} =\{\f_a(x); x\in\S_{a,(p,c,s)}\} = \beta^{-1}( \gamma(p) + \F_c ) \subseteq \F_a.$$
For a direction $\tau\in\SS^1$, we define $v_{a,(p,c,s)}(\tau)= \min\{ \Re(\tau z ); z\in \F_{a, (p,c,s)} \}$ and  $E_{a,(p,c,s)}(\tau) = E_a(\tau) \cap \F_{a,(p,c,s)}$.
\end{defn}

It is easy to see that $\F_a = \bigcup\F_{a,(p,c,s)}$ and that $v_a(\tau) =\min\{ v_{a,(p,c,s)}(\tau) \}$, where the union and the minimum are taken over the $(p, c, s) \in \A \times \A^* \times \A$ such that $\sigma(a) = pcs$.
\smallbreak

\begin{defn}\label{def:T_a}
	For each $a \in \A$ we define $\T_a$ as the set of directions
	$\tau\in\SS^1$ for which there exist distinct decompositions $\sigma(a) = pcs=\bar{p}\bar{c}\bar{s}$ in $\A^* \times \A \times \A^*$ with $E_{a,(p,c,s)}(\tau) \neq E_{a,(\bar p, \bar c, \bar s)}(\tau)$ and
	$ v_a(\tau) = v_{a,(p,c,s)(\tau)} = v_{a,(\bar p,\bar c, \bar s)}(\tau)$. 
	That is, $\T_a$ is
	the set of directions for which there exist distinct extreme points in $\F_a$ belonging to distinct subfractals $\F_{a,(p,c,s)}$ and $\F_{a,(\bar p,\bar c, \bar s)}$ (see Figure \ref{fig:Psi}). Put $\T = \bigcup_{a\in\A} \T_a$. 
\end{defn}

\begin{figure}
	\centering
	\includegraphics[width=0.55\textwidth]{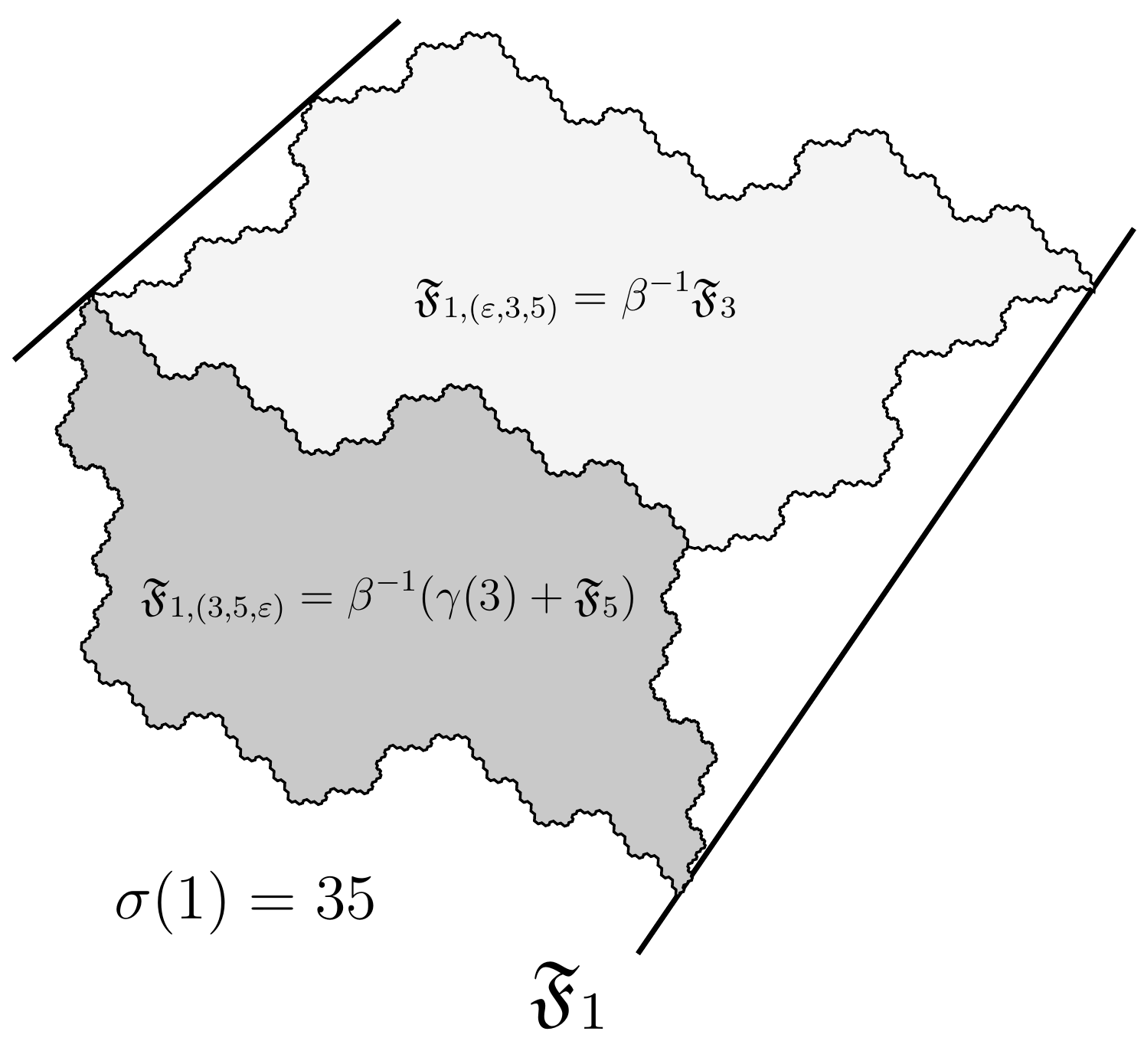}
	\caption{In the cubic Arnoux-Yoccoz map we show directions in $\Psi_1$.}
	\label{fig:Psi}
\end{figure}

\begin{lem}\label{lem:countable}
For each $a \in \A$ the set $\T_a$ is at most countable.
\end{lem}
\begin{proof}	
Let $a \in \A$. We will prove that the set $\Psi_a' = \{ \tau \in \SS^1; |E_a(\tau)| \geq 2 \}$, which clearly contains $\T_a$, is at most countable.
	
Let $Z \subseteq \C$ and $\tau \in \SS^1$. We define the set of extreme points of $Z$ for the direction $\tau$ as $E_Z(\tau) = \{ z \in Z; \Re(\tau z) = \min \{\Re(\tau z'); z' \in Z \}\}$. Similarly, we define the set of directions in $\SS^1$ for which there exist at least two distinct extreme points as $\T_Z'(\tau) = \{ \tau \in \SS^1; |E_Z(\tau)| \geq 2 \}$. We denote the convex hull of $Z$ by $\mathrm{conv}(Z)$.

We have that $E_a(\tau) = E_{\F_a}(\tau)$, $\T_a'(\tau) = \T_{\F_a}'(\tau)$, $E_{\mathrm{conv}(\F_a)}(\tau) = \mathrm{conv}(E_a(\tau))$ and that $\T_{\mathrm{conv}(\F_a)}' = \T_a'$, since $|E_a(\tau)| \geq 2$ if and only if $|\mathrm{conv}(E_a(\tau))| \geq 2$.

Since $\mathrm{conv}(\F_a)$ is convex and compact, it is either a point, a line segment or homeomorphic to a closed disk. Since in the first two cases $\T_{\mathrm{conv}(\F_a)}'$ is finite, we can assume that the third case holds. Let $\Phi \colon \SS^1 \to \partial \mathrm{conv}(\F_a)$ be an homeomorphism between $\SS^1$ and the boundary of $\mathrm{conv}(\F_a)$. We have that $\tau \in \T_a'$ if and only if the map $\Re(\tau \Phi)\colon \SS^1 \to\R$ is constant on an open interval of $\SS^1$. Since such open intervals must be disjoint for distinct $\tau$'s and a family of disjoint open subsets of $\SS^1$ is at most countable, we obtain that $\T_a'$ is at most countable.
\end{proof}

\subsection{Minimal sequences for \texorpdfstring{$\bm{\gamma}$}{gamma} and limit extreme points} \label{sec:limit-extreme}

Points in $\S_a$ may be constructed by ``reversing'' the prefix-suffix decomposition of elements in $\Omega_T$. Moreover, minimal sequences in $\Omega_T$ for $\gamma$ produce, by reversing its prefix-suffix decomposition, 
extreme points for any direction. We will make this statement precise in this section.

Consider a minimal sequence 
$\omega \in \Omega_T$ for $\gamma$ and let $(p_m, c_m, s_m)_{m \geq 1}$ be its prefix-suffix decomposition. Fix some $\tau\in\SS^1$ and recall that $\beta_0=\beta/|\beta|$ is not a root of unity.
Let $(n_k)_{k \geq 1}$ be an increasing sequence of integers such that $\beta_0^{n_k} \to \tau \in \SS^1$ and $c_{n_k}=a \in \A$ for every $k \geq 1$. For each $k\geq 1$ consider 
a point $x_{n_k} \in \S_a$ starting with   
$$(p_{n_k - 1}, c_{n_k - 1}, s_{n_k - 1})(p_{n_k - 2} c_{n_k - 2}, s_{n_k - 2})\ldots
(p_{0}, c_{0}, s_{0})$$ 
and such that $p_m^{x_{n_k}} = \varepsilon$ for any $m \geq n_k + 1$. This is a sequence obtained by reversing the indexes of finitely many elements of the prefix-suffix decomposition of $\omega$. Assume $(x_{n_k})_{k\geq 1}$ converges to $x_\infty \in \S_a$. We will show that $x_\infty$ is the representation of an extreme point in $\F_a$ for the direction $\tau$. In fact we can prove a little more. For this we need to introduce the following concept.

\begin{defn}\label{def:limit extreme}
	A limit extreme point is an extreme point that has a representation $x\in\S_a$ with the following property: for each $j \geq 1$ there exists $y_j \in\S_{a_j}$ for some $a_j \in\A$ such that
	\begin{enumerate}[label=(\roman{*})]
		\item $y_j$ is a representation of an extreme point for the direction $\beta_0^j\tau$;
		\item $c_j^{y_j}=a$;
		\item $x  = S^{j}(y_j)$, where $S$ is the left shift map on $\bigcup_{a \in \A} \S_a$.
	\end{enumerate}

	The set of limit extreme points will be denoted by $E_a^*(\tau)$.
\end{defn}

\begin{lem}
\label{lem:reverse-trajectory}
Any limit $x_\infty \in \S_a$, constructed as above,  is a representation of a limit extreme point in $E_a^*(\tau)$.
\end{lem}
\begin{proof}
Assume the sequence constructed as before $(x_{n_k})_{k \geq 1}$ converges to $x_\infty \in \S_a$. For each integer $j \geq 0$, let $y_{n_k}^j$ be the sequence in $\S_{c_{n_k + j}}$ starting by
$$(p_{n_k + j - 1}, c_{n_k + j - 1}, s_{n_k + j - 1})(p_{n_k + j - 2} c_{n_k + j - 2}, s_{n_k + j - 2})\ldots
(p_{0}, c_{0}, s_{0}),$$
and such that $p_m^{y_{n_k}^j} = \varepsilon$ for every $m \geq n_k + j + 1$. That is, $y_{n_k}^j$ is the sequence obtained by reversing the first $n_k+j$ elements of the prefix-suffix decomposition $(p_m, c_m, s_m)_{m\ge 0}$ of $\omega$. Clearly, $S^j(y_{n_k}^j) = x_{n_k}$. By taking a subsequence we may assume that $c_{n_k+j} = a_j$ for every $k \geq 1$ and that $y_{n_k}^j$ converges to $y_j$ when $k \to \infty$. We clearly have that $c_j^{y_j}=a$ and $S^j(y_j) = x_\infty$.

We claim that both $x_\infty$ and $y_j$ are representations of extreme points for the letters $a$ and $a_j$ and the directions $\tau$ and $\beta_0^j \tau$ respectively. Indeed, since $\omega$ is a minimal sequence for $\gamma$, by Lemma \ref{lem:minimal-leq} we have that $x_{n_k}$ satisfies the hypotheses of Lemma \ref{lem:minimo n} (taking $\tau=1$ in the lemma), so $v_a^{(n_k)}(\beta_0^{n_k}) = \Re(\beta_0^{n_k} \f_a(x_{n_k}))$. Thus for every $x\in \S_a$ we have that 
$\Re(\beta_0^{n_k} \f_a(x_{n_k})) \leq \Re(\beta_0^{n_k} \f_a^{(n_k)}(x))$. 
Finally, taking the limit when $k \to \infty$ we get that 
$\Re(\tau \f_a(x_{\infty}))\leq \Re(\tau \f_a(x))$ for any $x\in \S_a$, so 
$x_\infty$ is an extreme point in $E_a(\tau)$. The proof for $y_j$ is similar.
\end{proof}
 
The previous lemma gives us the following important corollary:
\begin{cor}\label{cor:many limit extreme}
For every $\tau \in \SS^1$ there exists $a \in \A$ such that $E_a^*(\tau)$ is nonempty. Moreover, if $E_a^*(\tau)$ is nonempty for some $a \in \A$ and $\tau \in \SS^1$, then 
$E_a^*(\tau) = E_a(\tau)$
\end{cor}
\begin{proof}
Take a minimal sequence $\omega \in \Omega_T$ and $\tau\in\SS^1$. By reversing the prefix-suffix decomposition of 
$\omega$ as above on an appropriate sequence $(n_k)_{k \geq 1}$ such that $\beta_0^{n_k}\to \tau$, we obtain by previous lemma that $E_a^*(\tau) \neq \varnothing$ for some $a \in \A$.	
	Now, if $E_a^*(\tau)$ is nonempty for some $a \in \A$ and $\tau \in \SS^1$, let $x \in \S_a$ be a representation of a limit extreme point in $E_a^*(\tau)$, for every $j\geq 1$ let $y_j \in \S_{a_j}$ be the sequence from the definition of limit extreme point and let $x' \in \S_a$ be the representation of some extreme point in $E_a(\tau)$. For every $j\geq 1$ we define $y_j' \in \S_a$ to coincide with $y_j$ in its first $j$ coordinates and with $x'$ in the rest of its coordinates. A direct computation shows that $y_j' \in \S_{a_j}$ and that it is the representation of an extreme point in $E_{a_j}(\beta_0^j \tau)$ such that $S^j(y_j') = x'$ and $c_j^{y'_j}=a$.
\end{proof}
Even though for every $a \in \A$ and $\tau \in \SS^1$ one has that $E_a(\tau)$ is nonempty, in general $E_a^*(\tau)$ may be empty. Indeed, the Arnoux-Yoccoz fractals contain extreme points that are not limit extreme points.

Another interesting consequence of Lemma \ref{lem:reverse-trajectory} is the following:

\begin{lem}\label{lem:no periodic is minimal}
If $\omega \in \Omega_T$ has an eventually periodic prefix-suffix decomposition, then it cannot be minimal for any eigenvector $\gamma$ associated with $\beta$. 
\end{lem}
\begin{proof}
If it were, then using the reversing procedure described above Lemma \ref{lem:reverse-trajectory} we could construct a limit extreme point with periodic representation. This contradicts Lemma \ref{lem:minimal-not-periodic}.
\end{proof}

\section{Unique representation property}\label{sec:URP}
We continue with notations of previous sections. So, $T$ is a self-similar interval exchange map, which is self-induced on the interval $[0,\alpha)$, and $\beta$ is an eigenvalue of $M$ with $|\beta| > 1$ such that $\beta/|\beta|$ is not a root of unity. Consider an eigenvector $\gamma$ of $M$ associated to $\beta$. 
Recall that $x\in\S_a$ is a representation of $z \in \F_a$ if $z=\f_a(x)$.

\begin{defn}
We say that $T$ satisfies the \emph{unique representation property} (u.r.p.)\ for $\beta$ and the eigenvector $\gamma$ if every extreme point of the associated fractals has a unique representation. We say that $T$ satisfies the \emph{weak unique representation property} (weak u.r.p.)\ for $\beta$ and the eigenvector $\gamma$ if every limit extreme point of the associated fractals has a unique representation.
\end{defn}

By Corollary \ref{cor:many limit extreme}, if $T$ satisfies the weak u.r.p.\ for $\beta$ and the eigenvector $\gamma$, then for all $a \in \A$ and $\tau \in \SS^1$ such that $E_a^*(\tau)$ in nonempty, we have that every extreme point of $\F_a$ for the direction $\tau$ has a unique representation. Indeed, in this case $E_a(\tau) = E_a^*(\tau)$. In particular, if $\sigma(a) = pcs = \bar{p}\bar{c}\bar{s}$ for distinct decompositions in $\A^* \times \A \times \A^*$, then $E_{a, (p,c,s)}(\tau) \cap E_{a, (\bar p,\bar c,\bar s)}(\tau)$ is empty.

Before continuing we need to comment the dependence of the previous concepts on the eigenvector $\gamma$. Until now, we have fixed an eigenvector $\gamma$ associated with $\beta$ and defined
the fractal set $\F_a $, its extreme and limit extreme points for a given direction $E_a(\tau)$ and  $E^*_a(\tau)$, and the directions with extreme points in distinct subfractals $\T_a$. These objects clearly depend on the choice of $\gamma$. We will temporarily make this dependence explicit by writing $\F_a(\gamma)$, $E_a(\gamma, \tau)$, $E^*_a(\gamma, \tau)$ and $\T_a(\gamma)$ respectively. Our main concern is to see how these concepts vary in the one dimensional space generated by $\gamma$. The following relations follow easily from definitions:

\begin{lem}\label{lem:homotecy}
	Let $\gamma, z\gamma \in \C^\A$ be eigenvectors associated with $\beta$, with $z \in \C \setminus \{0\}$. Let $a \in \A$ and $\tau \in \SS^1$. Then,
	\begin{enumerate}[label=(\roman{*})]
		\item $\F_a(z\gamma) = z \F_a(\gamma)$;
		\item $E_a(z\gamma, \tau) = z E_a(\gamma, z_0 \tau)$, $E^*_a(z\gamma, \tau) = z E^*_a(\gamma, z_0 \tau)$;
		\item $\T_a(z\gamma) = z_0^{-1} \T_a(\gamma)$,
	\end{enumerate}
	where $z_0 = z/|z|$.
\end{lem}

From this lemma we deduce that if the (weak) u.r.p.\ holds for an eigenvector $\gamma$ associated to $\beta$ then it holds for $z\gamma$ for all $z\in\C\setminus\{0\}$. If the eigenvalue $\beta$ is simple, then this condition is independent of the choice of the eigenvector, so in this case we can speak of the (weak) u.r.p.\ for $\beta$. We believe that there are some natural algebraic conditions that imply this fact.

The weak u.r.p.\ is in some sense analogous to the algebraic condition considered in the case $\beta$ is real. In fact, restating our definitions for the real case, Lemma 19 of \cite{persistence} is equivalent to the weak u.r.p.

We will prove in Section \ref{sec:example} that the cubic Arnoux-Yoccoz map satisfies the u.r.p.\ for some simple non real eigenvalue $\beta$.

\section{Proof of the main Theorem}\label{sec:main theorem}

We restate our main theorem for completeness. 
We continue with the notations of the three previous sections. 

\begin{cthm}{A}\label{teo:main}
Let $T$ be a self-similar i.e.m. Assume that $M$ has an eigenvalue $\beta$ with $|\beta|>1$ such that $\beta/|\beta|$ is not a root of unity, and that there exists an eigenvector $\Gamma$ for $\beta$ such that $T$ has the u.r.p.\ for $\beta$ and $\Gamma$. Then, for almost every $\gamma$ in the complex subspace generated by $\Gamma$, $\exp(-\Re(\gamma))$ can be realized as the slope vector of an affine i.e.m.\ which is semi-conjugate with $T$ and has wandering intervals. 
\end{cthm}
Of course, affine i.e.m.'s with wandering intervals cannot be conjugate with $T$, so the theorem asserts the existence of affine i.e.m.'s  strictly semi-conjugate with $T$. 
We remark that each of these affine i.e.m.'s is uniquely ergodic, since $T$ is.
\smallbreak

The weak u.r.p.\ is sufficient to prove Theorem~\ref{teo:main}. In fact, the proof of this theorem relies completely on Lemma~\ref{prop:main-prop}, which is also true under the weak u.r.p.\ as we point out at the end of the proof. 

The (weak) u.r.p.\ may seem technical and difficult to check for a specific map. In the next section we prove that this property is satisfied by the cubic Arnoux-Yoccoz i.e.m.
\smallbreak

As discussed in Section \ref{sec:strategy}, using the general strategy of \cite{cam-gut} the proof of Theorem \ref{teo:main} is a consequence of the following more technical statement that we prove in the next section.

\begin{thm}\label{teo:mainaux}
Under the same assumptions of Theorem \ref{teo:main}, there exists a number $\rho > 0$ such that, for almost all $\gamma$ in the complex subspace generated by $\Gamma$, every minimal sequence $\omega \in \Omega_T$ for $\gamma$ satisfies:
\begin{equation}\label{eq:goal}
\liminf_{n \to \infty} \frac{ \Re( \gamma_n(\omega)) }{n^\rho} > 0 \quad \text{ and } \quad  \liminf_{n \to \infty} \frac{ \Re( \gamma_{-n}(\omega)) }{n^\rho} > 0.
\end{equation}
\end{thm}

\subsection{Proof of Theorem \ref{teo:mainaux}}

\subsubsection{Differentiability of $v_a$} We start by describing some differentiability properties of the map $v_a \colon \SS^1 \to \R$. 

\begin{lem}\label{lem:derivatives}
Let $a \in \A$ and $\tau \in \SS^1$. One has
\begin{align*}
		\lim_{t \to 0^+} \frac{ v_a(\tau \exp(i t)) - v_a(\tau) }{t} &= -\Im( \tau e_a^+(\tau) ), \\
		\lim_{t \to 0^-} \frac{ v_a(\tau \exp(i t)) - v_a(\tau) }{t} &= -\Im( \tau e_a^-(\tau) ),
\end{align*}
where $e_a^+(\tau), e_a^-(\tau)$ are the points of $E_a(\tau)$ such that $z \mapsto \Im(\tau z)$ is maximal and minimal respectively.
\end{lem}

\begin{proof}
We will only prove the first equality, since the second one is analogous. Let $z = e_a^+(\tau), z_t \in E_a(\tau \exp(i t))$ and let us write $\tau z = r \exp(i \theta)$ and $\tau z_t = r_t \exp(i \theta_t)$ for $r, r_t > 0$, $\theta, \theta_t \in [0, 2\pi)$. We will assume that $0 < t < \pi/2$.
	
Since $z = e_a^+(\tau)$ and $z_t \in E_a(\tau \exp(i t))$, we have that $v_a(\tau) = \Re(\tau z)=r \cos(\theta)$, $v_a(\tau \exp(i t)) = \Re(\tau \exp(it) z_t)=r_t \cos(\theta_t + t)$ and $-\Im(\tau e_a^+(\tau)) = -r\sin(\theta)$. Thus, we have to prove that
$$
\lim_{t \to 0^+} \frac{r_t \cos(\theta_t + t) - r \cos(\theta)}{t} = -r\sin(\theta).
$$
Since $z_t \in E_a(\tau \exp(i t))$, we have that
$$r \cos(\theta + t) = \Re(\tau \exp(i t) z) \geq \Re(\tau \exp(i t) z_t) = r_t \cos(\theta_t + t).$$
Hence, 
\begin{equation}\label{eq:lat_derv_1}
\begin{split}
0 &\geq r_t \cos(\theta_t + t) - r \cos(\theta + t) \\
&= (r_t \cos(\theta_t) - r \cos(\theta) ) \cos(t) + (r \sin(\theta) - r_t \sin(\theta_t)) \sin(t).
\end{split}
\end{equation}
Moreover, since $z \in E_a(\tau)$ we have that $r \cos(\theta) = \Re(\tau z) \leq \Re( \tau z_t ) = r_t \cos(\theta_t)$. 
We conclude that $(r_t \cos(\theta_t) - r \cos(\theta) ) \cos(t) \geq 0$ for $t$ small enough.
Therefore, from $\sin(t)>0$ and \eqref{eq:lat_derv_1} we have that $r \sin(\theta) - r_t \sin(\theta_t)$ cannot be positive. 
This proves that 
$\Im(\tau z) = r \sin(\theta) \leq r_t \sin(\theta_t) = \Im( \tau z_t )$.

We claim that $\lim_{t \to 0^+} \Im(\tau z_t) = \Im(\tau z)$. Indeed, fix a real $\eta > 0$ and consider
$$u(t) = \min\{ \Re(\tau \exp(i t) z'); z' \in \F_a, \Im(\tau z') \geq \Im(\tau z) + \eta\}.$$
Since $z$ is in $E_a(\tau)$ and is chosen with the maximal possible value for $\Im(\tau z)$, we have that $\Re(\tau z) < u(0)$. The map $u \colon [0, \pi/2) \to \R$ is continuous, so there exists $t_0 > 0$ such that $\Re(\tau \exp(i t) z) < \Re(\tau \exp(i t) z')$ for every $0 \leq t < t_0$ and $z' \in \F_a$ such that $\Im(\tau z') \geq \Im(\tau z) + \eta$. Therefore, if  $0 \leq t < t_0$ and $\Im(\tau z_t) \geq \Im(\tau z) + \eta$, then 
$\Re(\tau \exp(i t) z) < \Re(\tau \exp(i t) z_t)$, which contradicts the fact that $z_t \in E_a(\tau \exp(i t))$. This shows that $\Im(\tau z_t) < \Im(\tau z) + \eta$ for $0 \leq t < t_0$, so we get the desired result by taking $\eta \to 0$.

The previous claim and \eqref{eq:lat_derv_1} imply that
$$
\lim_{t \to 0^+} \frac{r_t \cos(\theta_t + t) - r \cos(\theta + t)}{t} \geq \lim_{t \to 0^+} (\Im(\tau z_t ) - \Im(\tau z)) \frac{\sin(t)}{t} = 0.
$$
Finally, we write
$$
\frac{v_a(\tau \exp(i t)) - v_a(\tau)}{t} = \frac{r_t \cos(\theta_t + t) - r \cos(\theta + t)}{t} + \frac{ r \cos(\theta + t) - r \cos(\theta) }{t},
$$
and the result follows by taking $t \to 0^+$.
\end{proof}

\subsubsection{Good directions and good eigenvectors}

In order to prove the convergences in \eqref{eq:goal} we need to control the velocity at which $\beta_0^n$ approaches a number $\tau$ in $\SS^1$. 
In what follows $\llbracket \tau - \tau' \rrbracket$ is the natural distance between $\tau$ and $\tau'$ in $\SS^1$. We fix an eigenvector $\Gamma$ of $M$ for $\beta$. We will be interested in the complex subspace generated by $\Gamma$.

\begin{defn}\label{def:good}
A direction $\xi \in \SS^1$ is \emph{good} for $\Gamma$ if for every constant $A > 1$ and  every $\tau \in \T(\Gamma) = \bigcup_{a \in \A} \T_a(\Gamma)$ we have $\liminf_{n \to \infty}  A^n \llbracket \tau - \beta_0^n\xi \rrbracket > 0$.
\end{defn}

As shown by the next lemma, this property is generic.

\begin{lem}\label{lem:lebesguefull}
Almost every direction $\xi \in \SS^1$ is good for $\Gamma$.
\end{lem}

\begin{proof}
Let $A > 1$ and $\tau \in \SS^1$. We will first prove that
$$
K(A, \tau) = \left\{ \xi \in \SS^1; \liminf_{n \to \infty} A^n \llbracket \tau - \beta_0^n\xi \rrbracket > 0 \right\}
$$
has full Lebesgue measure. Consider the sets $B_n = \{ \xi\in\SS^1; \llbracket \tau - \beta_0^n\xi \rrbracket < A^{-n} \}$ for an integer $n \geq 0$. By the Borel-Cantelli Lemma, the Lebesgue measure of $\limsup_{n\to \infty} B_n$ is zero. That is, the set of $\xi$ which belong to infinitely many of the $B_n$'s has Lebesgue measure zero.
This implies that for a typical $\xi \in \SS^1$ there exists some $N \geq 1$ such that  $A^n \llbracket \tau - \beta_0^n\xi \rrbracket \geq 1$ if $n \ge N$. This proves the claim. 

Now, since by Lemma \ref{lem:countable} we have that $\T(\Gamma)$ is at most countable, the intersection  
$K(A) = \bigcap_{\tau \in \T(\Gamma)}$ $K(A, \tau)$ also has full Lebesgue measure. Finally, the intersection $K = \bigcap_{n \geq 1} K(1+1/n)$ has full Lebesgue measure and it is easy to see that every element of $K$ is a good direction.
\end{proof}

We can now define the eigenvectors for which our main result is valid.
\begin{defn}	
An eigenvector $\gamma$ of $M$ associated with $\beta$ is a \emph{good eigenvector} if for every $A>1$ and every $\tau\in\bigcup_{a\in \A} \T_a(\gamma)$ we have $\liminf_{n\to\infty} A^n \llbracket \tau - \beta_0^n \rrbracket > 0.$
\end{defn}

One has that if $\xi \in \SS^1$ is a good direction for $\Gamma$, then $\xi \Gamma$ is a good eigenvector by Lemma \ref{lem:homotecy}. Therefore, by Lemma \ref{lem:lebesguefull} we conclude that:
\begin{lem}\label{lem:almost-good-eigenvector}
If $\Gamma$ is an eigenvector of $M$ for $\beta$, then almost every vector of the complex subspace generated by $\Gamma$ is good.
\end{lem}

\subsubsection{Convergence in Theorem \ref{teo:mainaux}}
Assume the existence of an eigenvector $\Gamma$ for $\beta$ such that $T$ has the (weak) u.r.p.\ for $\beta$ and $\Gamma$. We show the convergence in \eqref{eq:goal} for the set of good eigenvectors in the subspace generated by $\Gamma$, which has full measure by the previous lemma. Fix a good eigenvector $\gamma$. As discussed at the end of Section~\ref{sec:URP}, $T$ also has the (weak) u.r.p.\ for $\beta$ and $\gamma$.

In the rest of the section, the sets $\F_a$, $E_a$, $\Psi_a$, etc.,\ are computed with respect to $\gamma$.

The proof of the next lemma and proposition closes the proof of Theorem \ref{teo:mainaux} and thus of Theorem \ref{teo:main} as explained in Section \ref{sec:strategy}.

\begin{figure}
	\centering
	\includegraphics[width=0.5\textwidth]{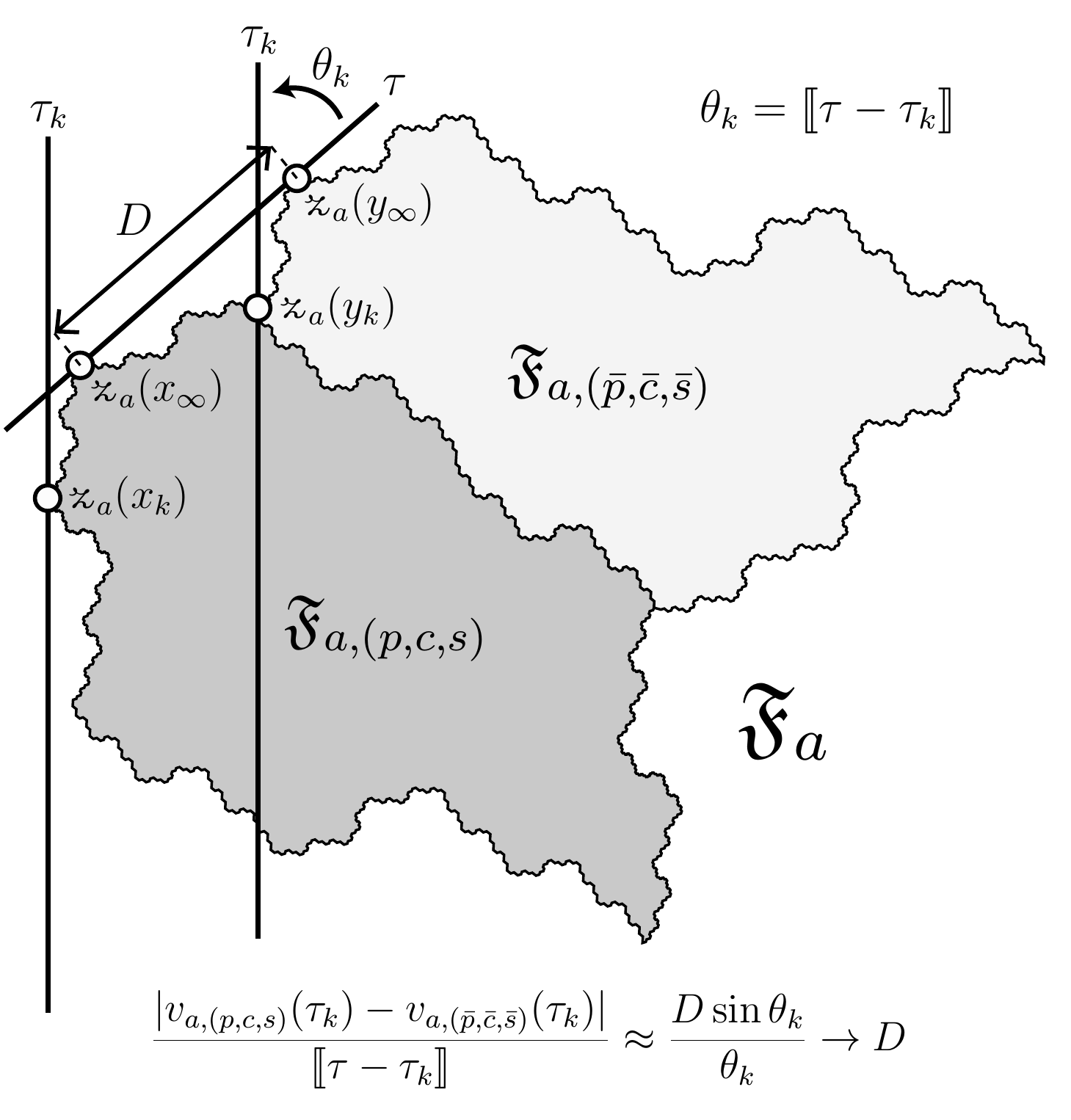}
	\caption{Illustration of the proof of Lemma \ref{lem:central}.}
	\label{fig:lemma}
\end{figure}

\begin{lem}\label{lem:central} Let $a \in \A$ and $\tau\in \Psi_a$. Assume that $(p,c,s), (\bar p,\bar c,\bar s)$ are such that $E_{a,(p,c,s)}(\tau)$ and $E_{a,(\bar p,\bar c,\bar s)}(\tau)$ are nonempty but $E_{a,(p,c,s)}(\tau) \cap E_{a,(\bar p,\bar c,\bar s)}(\tau) = \varnothing$. Put
$$ D =\min\{ |z - z'|; z\in E_{a,(p,c,s)}(\tau), z'\in E_{a,(\bar p,\bar c,\bar s)}(\tau)\}.$$
We have that $D>0$ and if $(\tau_k)_{k\ge 1}$ is a sequence in $\SS^1$ such that  $\tau_k \to \tau$ when $k\to\infty$,
	then 
	$$\liminf_{k\to\infty} \frac{|v_{a,(p,c,s)}(\tau_k) - v_{a,(\bar p,\bar c,\bar s)}(\tau_k)|}{\llbracket \tau - \tau_k \rrbracket} \ge D.$$
\end{lem}

\begin{proof}
See Figure \ref{fig:lemma} for an insight into the proof, which is in fact a little technical. 

The property $D>0$ is consequence of the fact that $E_{a,(\bar p,\bar c,\bar s)}(\tau)$ and $E_{a,(p,c,s)}(\tau)$ are nonempty, disjoint and compact.  

Let $(x_k)_{k\geq 1}, (y_k)_{k\geq 1}$ be sequences in $\S_{a,(p,c,s)}$ and $\S_{a,(\bar p,\bar c,\bar s)}$ respectively, such that $x_k\to x_\infty$ and $y_k\to y_\infty$ when $k\to\infty$ for some sequences in $x_\infty \in \S_{a,(p,c,s)}$ and $y_\infty \in \S_{a,(\bar p,\bar c,\bar s)}$ and 
$$v_{a,(p,c,s)}(\tau_k) =\Re(\tau_k\f_a(x_k)), \quad v_{a,(\bar p,\bar c,\bar s)}(\tau_k)= \Re(\tau_k\f_a(y_k))$$ for all $k\ge 1$. We remark that $\f_a(x_k)$ and $\f_a(y_k)$ attain the minimum for the direction $\tau_k$ and the subfractals $\F_{a,(p,c,s)}$ and $\F_{a,(\bar p,\bar c,\bar s)}$ respectively, but not necessarily for $\F_a$, as Figure~\ref{fig:lemma} illustrates.

By continuity and the fact that $E_{a,(p,c,s)}(\tau)$ and $E_{a,(\bar p,\bar c,\bar s)}(\tau)$ are nonempty by hypotheses, we have $\f_a(x_\infty) \in E_{a,(p,c,s)}(\tau)$ and $\f_a(y_\infty) \in E_{a,(\bar p,\bar c,\bar s)}(\tau)$. Therefore, \begin{equation}\label{eq:v(a;tau)}
 v_a(\tau) =\Re(\tau \f_a(x_\infty)) = \Re(\tau \f_a(y_\infty)). 
 \end{equation}
From the definition of $D$,
\begin{equation}\label{eq:imaginary not zero}
|\Im(\tau (\f_a(x_\infty) - \f_a(y_\infty))| \ge D.
\end{equation}
	
Since $x_k \to x_\infty$ and $y_k \to y_\infty$, there exists an increasing sequence $(n_k)_{k\ge 1}$ such that $(p_m^{x_k}, c_m^{x_k}, s_m^{x_k}) = (p_m^{x_\infty}, c_m^{x_\infty}, s_m^{x_\infty})$ and $(p_m^{y_k}, c_m^{y_k}, s_m^{y_k}) = (p_m^{y_\infty}, c_m^{y_\infty}, s_m^{y_\infty})$
for every $1\le m \le n_k$. Without loss of generality, we may assume that $c_{n_k}^{x_k} = c_{n_k}^{x_\infty} = b$ and that $c_{n_k}^{y_k} = c_{n_k}^{y_\infty} = \bar b$ for every $k\ge 1$. 
	
Let $S$ be the left shift in $\bigcup_{a\in\A} \S_a$. Recall that, from the continuation property in Lemma \ref{lem:continuation}, if $x\in \S_a$ represents an extreme point for the direction $\tau$ then
$ S^m(x) \in \S_{c_m^x}$ represents an extreme point for the direction $\beta_0^{-m}\tau$, for every integer $m \geq 0$. Consequently, we obtain that
$  v_{b} (\beta_0^{-n_k}\tau_k)= \Re (\beta_0^{-n_k}\tau_k\f_{b}(S^{n_k}(x_k)) )$  
and $v_{b} (\beta_0^{-n_k}\tau) = \Re( \beta_0^{-n_k}\tau
\f_{b}(S^{n_k}(x_\infty)))$ for every $k \geq 1$.
		
Therefore, we can write
\begin{equation}\label{eq:forma2}
\begin{split}
\Re(\tau \f_a(x_\infty)) -	\Re(\tau_k \f_a(x_k) )  & = \Re ( (\tau - \tau_k) ( \beta^{-1} \gamma(p^{x_\infty}_1)+ \dots +\beta^{-n_k} \gamma(p^{x_\infty}_{n_k}) ) ) \\ &+ |\beta|^{-n_k} ( v_{b}(\beta_0^{-n_k}\tau) - v_{b}(\beta_0^{-n_k}\tau_k) ).
\end{split} \end{equation}
Analogously,
	
\begin{equation}\label{eq:forma1}
\begin{split}
	\Re(\tau \f_a(y_\infty)) -	\Re(\tau_k \f_a(y_k))  & = \Re( (\tau - \tau_k)( \beta^{-1} \gamma(p^{y_\infty}_1)+ \dots +\beta^{-n_k} \gamma(p^{y_\infty}_{n_k}) )) \\ &+ |\beta|^{-n_k} ( v_{\bar b}(\beta_0^{-n_k}\tau ) - v_{\bar b}(\beta_0^{-n_k}\tau_k) ).
	\end{split} \end{equation}
	
Thus, by taking \eqref{eq:forma1}${}-{}$\eqref{eq:forma2}, multiplying by $\llbracket \tau - \tau_k \rrbracket^{-1}$ and using \eqref{eq:v(a;tau)} we obtain
	
	\begin{equation}\label{eq:derivada}
	\begin{split}
	\frac{v_{a,(p,c,s)}(\tau_k) - v_{a,(\bar p,\bar c,\bar s)}(\tau_k)}{\llbracket \tau - \tau_k \rrbracket}&=   \Re\left(\frac{\tau-\tau_k}{\llbracket \tau - \tau_k \rrbracket} \left( \f_a^{(n_k)} (y_\infty) - \f_a^{(n_k)} (x_\infty) \right) \right) \\
	&\quad -|\beta|^{-n_k} \left( \frac{v_{b}(\beta_0^{-n_k}\tau) - v_{b}(\beta_0^{-n_k}\tau_k)}{\llbracket \tau - \tau_k \rrbracket} \right)
	\\ & \quad + |\beta|^{-n_k} \left( \frac{v_{\bar b}(\beta_0^{-n_k}\tau) - v_{\bar b}(\beta_0^{-n_k}\tau_k)}{\llbracket \tau - \tau_k \rrbracket} \right) .
	\end{split}
	\end{equation}

Since 
$\llbracket \beta_0^{-n_k} \tau - \beta_0^{-n_k}\tau_k \rrbracket = \llbracket  \tau - \tau_k \rrbracket$, Lemma~\ref{lem:derivatives} implies that the quotients
$$\frac{v_{b}(\beta_0^{-n_k}\tau) - v_{b}(\beta_0^{-n_k}\tau_k)}{\llbracket \tau - \tau_k \rrbracket} \quad \text{ and } \quad \frac{v_{\bar b}(\beta_0^{-n_k}\tau) - v_{\bar b}(\beta_0^{-n_k}\tau_k)}{\llbracket \tau - \tau_k \rrbracket}$$
remain bounded for every $k\ge 1$, so the last two terms in the previous equality goes to $0$ when $k\to \infty$. Now, if $\tau,\tau'$ belong to $\SS^1$ then $$\lim\limits_{\tau^\prime\to\tau} \frac{\tau-\tau^\prime}{\llbracket \tau -\tau^\prime \rrbracket } = i\tau$$ and therefore 
$$\Re\left( \frac{\tau-\tau_k}{\llbracket \tau - \tau_k \rrbracket} \left(  \f_a^{(n_k)} (y_\infty) - \f_a^{(n_k)} (x_\infty)\right)\right)$$
converges to
$\Re( i\tau( \f_a(y_\infty) - \f_a(x_\infty) ) )$ when $k \to \infty$, which is the imaginary part of $\tau( \f_a(y_\infty) - \f_a(x_\infty) )  $. Finally, we obtain that when $k \to \infty$:
$$\frac{|v_{a,(p,c,s)}(\tau_k) - v_{a,(\bar p, \bar c, \bar s)}(\tau_k)|}{\llbracket \tau -\tau_k \rrbracket} \to |\Im(\tau (\f_a(x_\infty) - \f_a(y_\infty) )| \ge D,$$
where in the last inequality we have used \eqref{eq:imaginary not zero}.
\end{proof}

\begin{prop}\label{prop:main-prop}
Assume the hypotheses of Theorem~\ref{teo:main}. Consider $\omega \in \Omega_T$ a minimal sequence for $\gamma$. Let $\eta > 0$ with $|\beta| - \eta > 1$ and let $\rho = \frac{\log(|\beta| - \eta)}{\log(\alpha^{-1} + \eta)} > 0$. Then,
\begin{equation}\label{eq:mainlemma}
\liminf_{n \to \infty} \frac{ \Re( \gamma_n(\omega)) }{n^\rho} > 0 \quad \text{ and } \quad  \liminf_{n \to \infty} \frac{ \Re( \gamma_{-n}(\omega)) }{n^\rho} > 0. 
\end{equation}
\end{prop}
	
\begin{proof} We will only prove the first inequality, since the other one is similar. 

We denote the prefix-suffix decomposition of $\omega$ by $(p_m^{\omega}, c_m^{\omega}, s_m^{\omega})_{m \geq 0}$. Let us assume that there exists an increasing sequence of positive integers $(n_k)_{k \geq 1}$ such that
\begin{equation}\label{eq:hypothesis}
\lim_{k \to \infty} \frac{\Re( \gamma_{n_k}(\omega) )}{n_k^\rho} = 0.
\end{equation}
Let $(p_m^{\omega_k}, c_m^{\omega_k}, s_m^{\omega_k})_{m \geq 0}$ be the prefix-suffix decomposition of $\omega_k = S^{n_k}(\omega)$, where as usual $S$ is the left shift map on the corresponding subshift. 
We start by showing that $\omega$ must have infinitely many nonempty suffixes in its prefix-suffix decomposition Indeed, assume by contradiction that $s_{n_0 + m}^\omega = \varepsilon$ for some integer $n_0 \geq 0$ and every $m \geq 0$. We will show that $(p_m^{\omega}, c_m^{\omega}, s_m^{\omega})_{m \geq 0}$ is eventually periodic, which contradicts Lemma \ref{lem:no periodic is minimal}. We have that $\sigma(c_{n_0 + m + 1}^\omega) = p_{n_0 + m}^\omega c_{n_0 + m}^\omega$ for every $m \geq 0$. Then, for every $m \geq 0$, the value of $c_{n_0 + m + 1}^{\omega}$ determines a unique possible value for $p_{n_0 + m}^{\omega}$ and $c_{n_0 + m}^{\omega}$. By induction, it is easy to see that $(p_m^{\omega}, c_m^{\omega}, s_m^{\omega})_{m \geq n_0}$ is periodic.

Let $N_k$ be the first integer such that $(p_m^{\omega}, c_m^{\omega}, s_m^{\omega})_{m \geq N_k} = (p_m^{\omega_k}, c_m^{\omega_k}, s_m^{\omega_k})_{m \geq N_k}$. By taking a subsequence we can assume that $(N_k)_{k \geq 1}$ is an increasing sequence of integers. Moreover, we may assume that for $k\ge 1$:
\begin{enumerate}[label=(\roman{*})]
\item $c_{N_k}^{\omega} = c_{N_k}^{\omega_k} = a$;
\item $(p_{N_k - 1}^{\omega}, c_{N_k - 1}^{\omega}, s_{N_k - 1}^{\omega}) = (p, c, s)$;
\item $(p_{N_k - 1}^{\omega_k}, c_{N_k - 1}^{\omega_k}, s_{N_k - 1}^{\omega_k}) = (\bar p, \bar c, \bar s)$;
\item $pc$ is a prefix of $q$;
\item $\lim_{k \to \infty} \beta_0^{N_k} = \tau \in \SS^1$.
\end{enumerate}
Since $pc$ is a prefix of $q$, we have that
\begin{equation}\label{eq:teo1_1}
\sigma^{N_k - 1}(p_{N_k - 1}^{\omega_k}) \ldots p_0^{\omega_k} = \sigma^{N_k - 1}(p_{N_k - 1}^{\omega}) \ldots p_0^{\omega}\omega_0 \ldots \omega_{n_k - 1}
\end{equation}
	for every $k \geq 1$.
	
We will now proceed to reverse the indexes of the prefix-suffix decompositions of $\omega$ and $\omega_k$ in order to obtain sequences in $\S_a $. Let $(x_{N_k})_{k \geq 1}$ and $(y_{N_k})_{k \geq 1}$ be the sequences in $\S_a$, obtained by reversing the coordinates of $(p_m^\omega, c_m^\omega, s_m^\omega)_{m \geq 0}$ and $(p_{m}^{\omega_k}, c_{m}^{\omega_k}, p_{m}^{\omega_k} )_{m\ge 0}$  up to the $(N_k-1)$-th coordinate and such that $p_m^{x_{N_k}} = p_m^{y_{N_k}}= \varepsilon$ for each $m \geq N_k$, as detailed at the beginning of Section \ref{sec:limit-extreme}. By the assumptions above, $x_{N_k} \in \S_{a,(p,c,s)}$ and $y_{N_k} \in \S_{a,(\bar p, \bar c, \bar s)}$ for every $k \geq 1$. 

Without loss of generality, we will assume that $x_{N_k}$ converges to $x_\infty \in \S_{a,(p,c,s)}$, which is the representation of a limit extreme point in $E^*_a(\tau)$ by Lemma \ref{lem:reverse-trajectory}. We will show that any limit point of $y_{N_k}$ in $\S_{a,(\bar p, \bar c, \bar s)}$ is the representation of an extreme point in $E_a(\tau)$ and therefore that $\tau$ belongs to $\Psi_a$.
\medbreak
	
Applying $\gamma$ to \eqref{eq:teo1_1}, using the definitions of $x_{N_k}, y_{N_k}$ and multiplying by $|\beta|^{-N_k}$,  we get that for every $k \geq 1$:
$$
\beta_0^{N_k} \sum_{m = 1}^{N_k} \beta^{-m} \gamma( p_{m}^{y_{N_k}} ) 
=  \beta_0^{N_k} \sum_{m = 1}^{N_k} \beta^{-m} \gamma ( p_{m}^{x_{N_k}} ) +  |\beta|^{-N_k} \gamma_{n_k} (\omega).
$$

By taking real parts and rearranging the previous expression we obtain:
$$
\Re( \beta_0^{N_k} ( \f_a ( y_{N_k} ) - \f_a ( x_{N_k} ) ) ) = |\beta|^{-N_k} \Re( \gamma_{n_k}(\omega) ).
$$

Furthermore, $\Re(\beta_0^{N_k} \f_a(x_{N_k})) \leq \Re(\beta_0^{N_k} \f_a(y_{N_k}))$ by Lemma \ref{lem:minimo n}. Then we get
\begin{equation}\label{eq:teo1_2}
0 \leq \Re( \beta_0^{N_k} ( \f_a( y_{N_k} )- \f_a(x_{N_k} ) ) ) = |\beta|^{-N_k} \Re( \gamma_{n_k}(\omega) ).
\end{equation}
On the other hand, since $\omega_0\omega_1\ldots\omega_{n_k-1}$ is a subword of $\sigma^{N_k}(a)$ and $|\sigma^{N_k}(a)|$ grows as $\alpha^{-N_k}$ (recall that $\alpha^{-1}>1$ is the Perron-Frobenius eigenvalue of $M$), we have that $n_k \leq (\alpha^{-1}+\eta)^{N_k}$ for sufficiently large $k \geq 1$, where $\eta >0$ was given in the statement of the proposition. 
Therefore, by definition of $\rho$, 
$$
n_k^{-\rho}\geq (\alpha^{-1}+\eta)^{-N_k \rho} = (|\beta|-\eta)^{-N_k} \geq |\beta|^{-N_k}.
$$
From the assumption \eqref{eq:hypothesis}, we obtain that
$$
\lim_{k\to \infty} ( |\beta|-\eta )^{-N_k}  \Re( \gamma_{n_k}(\omega) ) = \lim_{k\to \infty} |\beta|^{-N_k}  \Re( \gamma_{n_k}(\omega) ) = 0.
$$
In particular, from equation~\eqref{eq:teo1_2} we obtain that any limit point $y_\infty$ of $y_{N_k}$ in $\S_{a,(\bar p, \bar c, \bar s)}$ is such that
$\f_a(y_\infty)$ is an extreme point for the direction $\tau = \lim_{k \to \infty} \beta_0^{N_k}$: $v_a(\tau) = \Re(\tau \f_a(y_\infty))$. Therefore, $\tau$ belongs to $\Psi_a$.

Amplifying equation \eqref{eq:teo1_2} by $A^{N_k}$, where $A = \frac{|\beta|}{|\beta| - \eta} \in (1, |\beta|)$, we find that
$$
0 \leq A^{N_k} \Re(\beta_0^{N_k} ( \f_a( y_{N_k} )- \f_a(x_{N_k} ) ) ) = (|\beta| - \eta)^{-N_k} \Re( \gamma_{n_k}(\omega) )
$$
for all sufficiently large $k$. Hence,
\begin{equation}\label{eq:contradicted}
\lim_{k \to \infty} A^{N_k} \Re(\beta_0^{N_k} ( \f_a( y_{N_k} )- \f_a(x_{N_k} ) ) ) = 0.
\end{equation}

We know from Lemma \ref{lem:minimo n}  that
$$ v_a^{(N_k)}(\beta_0^{N_k}) = v_{a,(p,c,s)}^{(N_k)}(\beta_0^{N_k}) = \Re(\beta_0^{N_k} \f_a (x_{N_k})).$$
We also know that $\Re(\beta_0^{N_k} \f_a (y_{N_k})) \geq v_{a,(\bar p, \bar c, \bar s)}^{(N_k)}(\beta_0^{N_k})$ and therefore that
\begin{equation} \label{eq:desig-1}
\Re(\beta_0^{N_k} ( \f_a( y_{N_k} )- \f_a(x_{N_k} ) ) ) \ge v_{a,(\bar p, \bar c, \bar s)}^{(N_k)}(\beta_0^{N_k}) - v_{a,(p,c,s)}^{(N_k)}(\beta_0^{N_k}) \ge 0.  \end{equation} 

On the other hand, since $x_\infty\in \S_{a,(p,c,s)}$  and $y_\infty\in\S_{a,(\bar p, \bar c, \bar s)}$ are representations of extreme points for $\tau$, the u.r.p.\ implies that $E_{a,(p,c,s)}(\tau)$ and $E_{a,(\bar p, \bar c, \bar s)}(\tau)$ are disjoint, so in particular $\f_a(x_\infty) \neq \f_a(y_\infty)$. We notice at this point that the weak u.r.p.\ is sufficient. Indeed the extreme points $\f_a(x_\infty) $ and $ \f_a(y_\infty)$ are also limit extreme points in $ E_a^*(\tau) $, so the weak u.r.p.\ implies that they are different. Also notice that, by Corollary \ref{cor:many limit extreme}, $E_a(\tau) = E_a^*(\tau)$.

Using Lemma~\ref{lem:minimum-exponential} we conclude that for each $k \geq 1$:
\begin{equation}\label{eq:min-exponential}
v_{a,(\bar p, \bar c, \bar s)}^{(N_k)}(\beta_0^{N_k}) - v_{a,(p,c,s)}^{(N_k)}(\beta_0^{N_k}) \geq  v_{a,(\bar p, \bar c, \bar s)}(\beta_0^{N_k}) - v_{a,(p,c,s)}(\beta_0^{N_k}) - 2C |\beta|^{-N_k}\end{equation}
for a constant $C > 0$ which does not depend on $k$.

Finally, by \eqref{eq:desig-1}, \eqref{eq:min-exponential} and Lemma \ref{lem:central} we conclude that
\begin{equation*}
\Re(\beta_0^{N_k} ( \f_a( y_{N_k} )- \f_a(x_{N_k} ) ) ) \ge
D \llbracket \tau - \beta_0^{N_k} \rrbracket - 2C|\beta|^{-N_k} \end{equation*}
for infinitely many $k \geq 1$. Since $\gamma$ is a good eigenvector for $\Gamma$ and $\tau\in\Psi_a$, by definition, $\liminf_{k\to\infty} A^{N_k} \llbracket \tau - \beta_0^{N_k} \rrbracket = \infty$. This contradicts \eqref{eq:contradicted}.
\end{proof}

\subsection{Proof of Theorem \ref{teo:galoisconjugate}}
The Theorem follows from Theorem \ref{teo:main0} if we prove that $\beta$ is a simple eigenvalue of $M$ and that $\beta/|\beta|$ is not a root of unity. These two facts can proved if either $\beta$ is Galois-conjugate with $\alpha^{-1}$, or if $\beta$ is Galois-conjugate with $\alpha$ and the self-similarity comes from Rauzy-Veech renormalizations, that is, if some iteration of the Rauzy-Veech algorithm on $T$ and the interval $[0,\alpha)$ returns to $T$ (see \cite{viana} for details about the Rauzy-Veech algorithm). This last condition is natural in the following sense: let $\pi$ be a vertex on a Rauzy class and consider a cycle starting at $\pi$ in which every letter wins and loses at least once.
Let $R$ be the matrix of such cycle. Then, $R$ is primitive and $\alpha^{-1}$ is its Perron-Frobenius eigenvalue. Also, $R \lambda = \alpha^{-1} \lambda$ for some positive eigenvector $\lambda$ whose coordinates add up to $1$. The i.e.m.\ with combinatorial data $(\pi, \lambda)$ is periodic for the Rauzy-Veech algorithm. Thus, this algorithm provides a simple way to construct self-similar i.e.m.

Let $\psi\colon \Q(\xi) \mapsto \Q(\beta)$ be the natural field isomorphism coming from the Galois-conjugacy, where $\xi$ is either $\alpha$ or $\alpha^{-1}$. The following Lemma proves what we need:

\begin{lem}\label{lem:simple eigenspaces}
Let $\beta$ be an eigenvalue of $M$ that is either Galois-conjugate with $\alpha^{-1}$, or is Galois-conjugate with $\alpha$ and $T$ is periodic for the Rauzy-Veech renormalization algorithm on $[0,\alpha)$. Then $\beta$ has algebraic and geometric multiplicity one. Moreover, if $\beta$ is not real, then $\beta/|\beta|$ is not a root of unity.
\end{lem}

\begin{proof}
If $\beta$ is Galois-conjugate with $\alpha$, which is the Perron-Frobenius eigenvalue of $M$, then it has multiplicity $1$. We will now prove that if it is Galois-conjugate with $\alpha^{-1}$ then it also has multiplicity $1$ when one further assumes that $T$ is periodic for the Rauzy-Veech algorithm. We will use the classical results of Veech in \cite{veech-metric}, which use this fact.

Recall that $R = M^t$ is the renormalization matrix coming from the Rauzy-Veech induction. There exists a (possibly degenerate) antisymmetric integer matrix $L^\pi$ such that
\begin{equation}\label{eq:symplectic}
R^t L^\pi R = M L^\pi R = L^\pi.
\end{equation}
Let $H(\pi) = L^\pi(\C^\A)$ and $N(\pi) = \ker L^\pi$. It is easy to check from \eqref{eq:symplectic} that $H(\pi)$ is invariant for $M$ and that $N(\pi)$ is invariant for $R$. The matrix $L^\pi$ is nondegenerate when restricted to $H(\pi)$ and, therefore, $R|_{H(\pi)}$ is symplectic. We will show that the eigenspaces of $\alpha, \alpha^{-1}$ and $\beta$ are contained in $H(\pi)$.

By Lemma 5.6 of \cite{veech-metric}, $R$ acts as a permutation on a basis of $N(\pi)$. Therefore, every eigenvalue of $R|_{N(\pi)}$ has modulus $1$. Let $V$ be the eigenspace of an eigenvalue $z$ for $R$ with $|z| \neq 1$. From \eqref{eq:symplectic}, it is easy to see that $L^\pi(V) \subseteq H(\pi)$ is contained in the eigenspace of $z^{-1}$ for $M$ and, since $V \cap N(\pi) = \{0\}$, that $L^{\pi}(V)$ is the entire eigenspace of $z^{-1}$ for $M$. Thus, the desired eigenspaces are contained in $H(\pi)$.

Let $p(t)$ be the characteristic polynomial of $M$ restricted to $H(\pi)$. We have that $p(t)= t^{|\A|} p(1/t)$ by symplecticity. The Galois-conjugacy implies that $\beta$ has the same algebraic multiplicity of $ \alpha$. Since $\alpha^{-1}$ is a simple root of $p(t)$, and so is $\alpha$ by the equality $p(t)= t^{|\A|} p(1/t)$. The first part of the Lemma is therefore proved.

Now, assume by contradiction that $ \beta $ is Galois-conjugate to $ \xi $ (with  $ \xi= \alpha $ or $ \alpha^{-1} $) and that $\beta^n$ is real for some integer $n$. Since $\bar{\beta}$ is also an eigenvalue of $M$, which is different from $\beta$, we have that $\beta^n = \bar{\beta}^n$ has algebraic multiplicity $2$ for $M^n$. Moreover, $\xi^n$ and $\beta^n$ are also Galois-conjugate. Indeed, if $q(t)$ is the minimal polynomial of $\alpha^n$, then $q(\beta^n) = q(\psi(\xi^n)) = \psi(q(\xi^n)) = 0$.
The matrix $M^n$ is also primitive and corresponds to the induced map of $ T $ on $ [0,\alpha^n) $. We can therefore replicate the proof of the first part of the lemma for $M^n$ and conclude that $\beta^n$ has algebraic multiplicity one, which is a contradiction.
\end{proof}

\section{The cubic Arnoux-Yoccoz map} \label{sec:example}
In this section we illustrate Theorem \ref{teo:main} in the cubic Arnoux-Yoccoz i.e.m.\ (or A-Y i.e.m.\ for simplicity). This map is self-similar in the sense of \cite{geometriciem} but not in the sense of \cite{cam-gut}, 
which is the notion we are following in Theorem \ref{teo:main}. However, it satisfies all the other hypotheses of the theorem. Although the main theorem is written for 
self-similar i.e.m.\ in the sense of \cite{cam-gut}, this precise notion plays a role in a very specific part of the proof of the theorem, which is ensuring that the symbolic system associated with the i.e.m.\ is substitutive. It is possible to prove this fact for the specific case of the A-Y i.e.m.

In addition, it is proved in \cite{geometriciem} that an induced system obtained from the A-Y i.e.m.\ with respect to a precise interval is a self-similar i.e.m.\ in the sense of \cite{cam-gut}, but the resulting substitution associated with the i.e.m.\ is unnecessarily complex to analyse. Nevertheless, some technical but not difficult modifications on the study that we will develop for the A-Y i.e.m.\ below allow to prove that the induced system also satisfies the hypotheses of Theorem \ref{teo:main}.

\subsection{A-Y i.e.m.}

\begin{figure}
	\centering
	\includegraphics[width=0.45\textwidth]{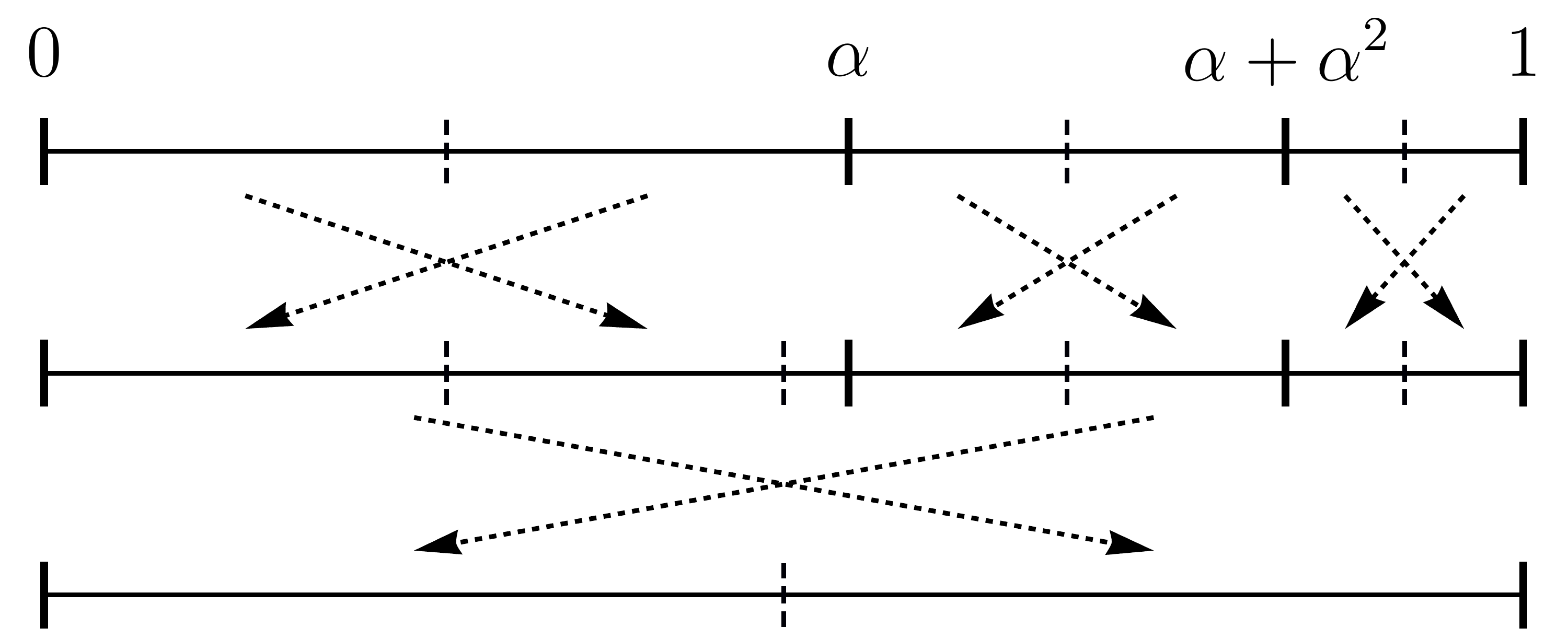}
	\caption{The compositions that produce the cubic Arnoux-Yoccoz i.e.m.\ $T$. The dashed lines show the midpoints of the respective intervals.}
	\label{fig:figureAY1}
\end{figure}

\begin{figure}
	\centering
	\includegraphics[width=0.6\textwidth]{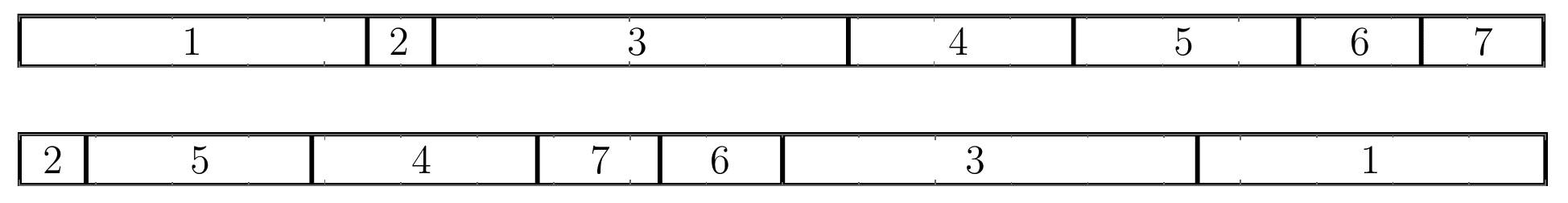}
	\caption{The cubic Arnoux-Yoccoz i.e.m.\ $T$.}
	\label{fig:figureAY2}
\end{figure}

Let $\alpha$ be the unique real number such that 
$\alpha + \alpha^2 + \alpha^3 = 1$ and let $G_{t_0,t_1}$ be the map exchanging both halves of the interval $[t_0, t_1)$ while preserving orientation. That is,
$$
G_{t_0, t_1}(t) = \begin{cases}
	t + (t_0 + t_1)/2 & t \in [t_0, (t_0+t_1)/2), \\
	t - (t_0 + t_1)/2 & t \in [(t_0+t_1)/2, t_1), \\
	t & t \notin [t_0, t_1).
\end{cases}
$$
Then, the A-Y i.e.m.\ is given by $T = G_{0,1} \circ G_{0,\alpha} \circ G_{\alpha, \alpha+\alpha^2} \circ G_{\alpha+\alpha^2,1}$ (see Figures \ref{fig:figureAY1} and \ref{fig:figureAY2} for clarity). Properties of $T$ were extensively discussed in \cite{geometricalmodels}. In particular, it is proved that the map $T$ is equal, up to rescaling and rotation, to the map induced on the interval $[0, \alpha)$ and, by considering an appropriate refinement of continuity intervals of $T$ into nine intervals, one may encode the relation of orbits by $T$ for this partition and the orbits of the induced system for the induced partition by the following substitution $\sigma$ on the alphabet $\A = \{1, \dotsc, 9\}$:
$$\arraycolsep=1.5pt
	\begin{array}{lllllllll}
		\sigma(1) & = & 35 &\qquad \sigma(4) & = & 17 &\qquad \sigma(7) & = & 29 \\
		\sigma(2) & = & 45 &\qquad \sigma(5) & = & 18 &\qquad \sigma(8) & = & 2   \\
		\sigma(3) & = & 46 &\qquad \sigma(6) & = & 19 &\qquad \sigma(9) & = & 3   \\
	\end{array}
$$
One then has that $\Omega_T = \Omega_\sigma$. It is easy to check that $\sigma$ is primitive. This solves the issue at the beginning of the section: the symbolic system is indeed substitutive.

Let $M$ be the matrix associated with the substitution $\sigma$, that is,

{\small
$$
	M = \begin{pmatrix}
		0 & 0 & 1 & 0 & 1 & 0 & 0 & 0 & 0 \\
		0 & 0 & 0 & 1 & 1 & 0 & 0 & 0 & 0 \\
		0 & 0 & 0 & 1 & 0 & 1 & 0 & 0 & 0 \\
		1 & 0 & 0 & 0 & 0 & 0 & 1 & 0 & 0 \\
		1 & 0 & 0 & 0 & 0 & 0 & 0 & 1 & 0 \\
		1 & 0 & 0 & 0 & 0 & 0 & 0 & 0 & 1 \\
		0 & 1 & 0 & 0 & 0 & 0 & 0 & 0 & 1 \\
		0 & 1 & 0 & 0 & 0 & 0 & 0 & 0 & 0 \\
		0 & 0 & 1 & 0 & 0 & 0 & 0 & 0 & 0
	\end{pmatrix} .
$$} 

Its characteristic polynomial is $(1-t^3)(t^3+t^2+t-1)(-t^3+t^2+t+1)$, where the last two factors are irreducible. The roots of $t^3+t^2+t-1$ are $\alpha$, $\beta$ and $\bar{\beta}$, whereas the roots of $-t^3+t^2+t+1$ are $\alpha^{-1}$, $\beta^{-1}$ and $\bar{\beta}^{-1}$, where $\alpha^{-1}$ is the Perron-Frobenius eigenvalue. We assume that $\beta$ is the eigenvalue with positive imaginary part. Numerically, 
$\beta \approx -0.771845 + 1.11514i$.  It is proved in \cite{rauzy} that $(\beta^{-1})^n$ is never real for any $n \in \Z$. Furthermore, the eigenvalue $\beta$ is simple and the corresponding eigenspace is generated by
$$\gamma =
(\beta^2+\beta+1, -\beta, -\beta, -\beta^2-\beta-1, \beta+1, \beta + 1, -\beta^2-\beta-2, -1, -1).$$
Therefore, it is enough to prove the u.r.p.\ for $\beta$.

In what follows $\beta$ and $\gamma$ are the corresponding eigenvalue and eigenvector of $M$ used in the previous sections.

\subsection{Fractals associated with the A-Y i.e.m.}\label{subsect:fractals associated}
Theorem 6.4 of \cite{geometricalmodels} shows that the fractals defined in Section \ref{sec:fractals} exist and satisfy $\F_2 = \F_3$, $\F_5 = \F_6$ and $\F_8 = \F_9$. An illustration of each fractal is given in Figure \ref{fig:AY}. The boundaries of these fractals can be constructed by combining pieces of the boundary of the standard tribonnacci fractal and can therefore be parametrized as we will see later. 

\subsubsection{Parametrization of the boundary of tribonnacci fractal} \label{sec:rauzy}
First we state some important properties of the tribonnacci fractal. We will follow \cite{rauzy} freely. 
Let $\NN$ be the set of sequences in $\{0, 1\}$ without three consecutive $1$'s. The (standard) tribonnacci fractal is defined by
$$
\RR = \left\{ \sum_{m \geq 3} \beta^{-m} a_{m-2}; (a_m)_{m \geq 1} \in \NN \right\}.
$$
For $(a_m)_{m \geq 1} \in \NN$, we define $\r((a_m)_{m \geq 1}) = \sum_{m \geq 3} \beta^{-m} a_{m-2}$. For $z \in \RR$, we say that $(a_m)_{m \geq 1} \in \NN$ is a \textit{representation} of $z$ if $z = \r((a_m)_{m \geq 1})$. Clearly, any sequence in $\NN$ starts with either $0$, $10$ or $11$. One has the following:
\begin{align*}
	\RR_0 = \beta^{-1} \RR &= \{ z \in \RR; z \text{ has a representation starting with } 0 \}, \\
	\RR_{10} = \beta^{-3} + \beta^{-2}\RR &= \{ z \in \RR; z \text{ has a representation starting with } 10 \}, \\
	\RR_{11} = \beta^{-3} + \beta^{-4} + \beta^{-3}\RR &= \{ z \in \RR; z \text{ has a representation starting with } 11 \}.
\end{align*}
These three subsets of $\RR$ are scalings, rotations and translations of $\RR$ and are disjoint except for a set of measure zero (see Figure \ref{fig:rauzy}).

Clearly $\NN$ is a subshift. As before $S$ is the shift map in $\NN$. For $z \in \RR$ and $(a_m)_{m \geq 1} \in \NN$ a representation of $z$, one has that $\r(S((a_m)_{m \geq 1})) \in \RR$ and
$$
\r(S((a_m)_{m \geq 1})) = \begin{cases}
	\beta z & \text{if } a_1 = 0 \\
	\beta (z - \beta^{-3}) & \text{if } a_1 = 1.
\end{cases}
$$
It is easy to see that the points in $\RR_{11}$ are mapped bijectively into $\RR_{10}$, that the points in $\RR_{10}$ are mapped bijectively to $\RR_0$, and that the points in $\RR_0$ are mapped bijectively to $\RR$.

\begin{figure}
	\centering
	\includegraphics[width=0.8\textwidth]{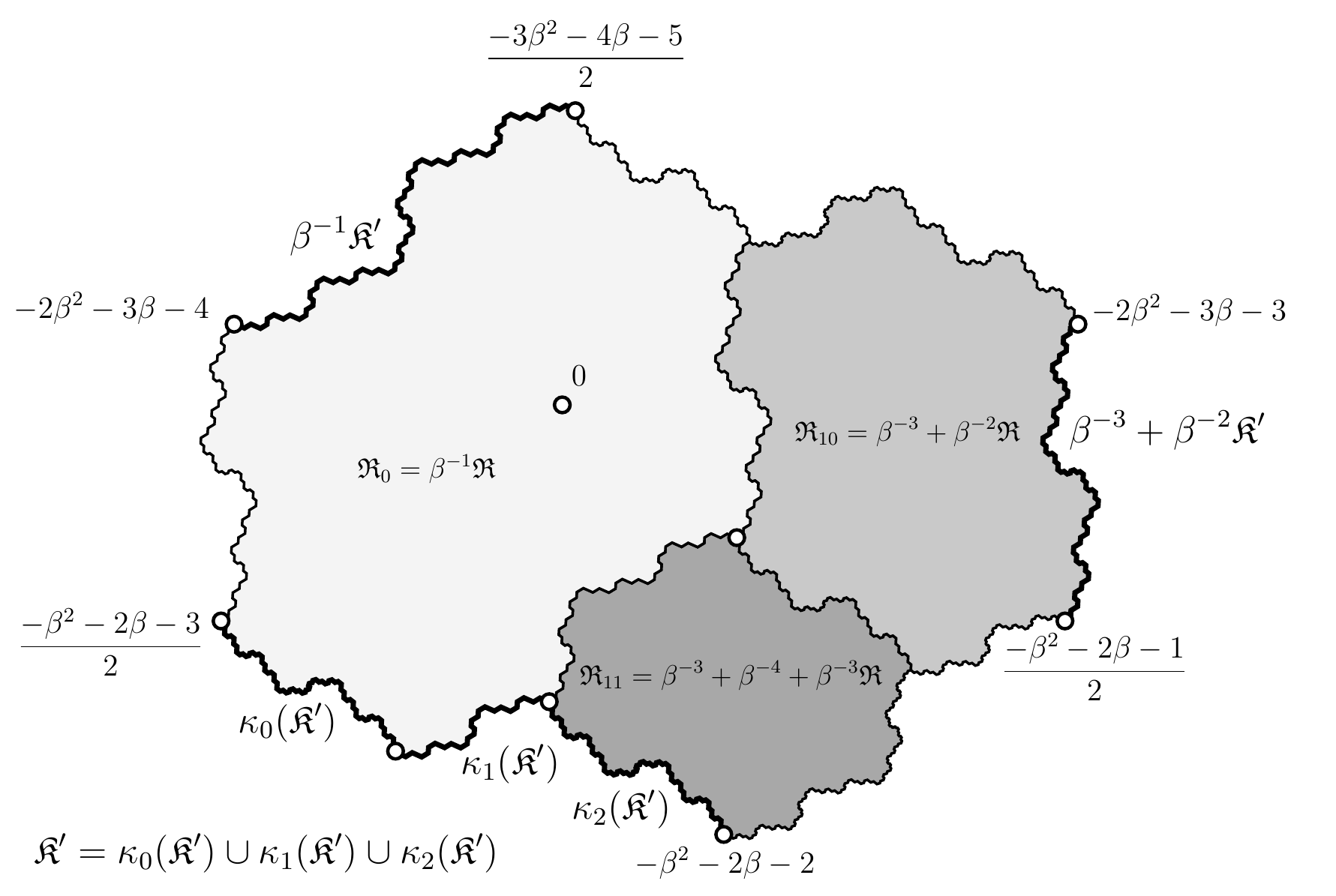}
	\caption{The tribonnacci fractal with some important features.}
	\label{fig:rauzy}
\end{figure}

Now, the parametrization of the boundary of the tribonnacci fractal is constructed as follows. Put $z_0 = \beta^{-4}/(1 - \beta^{-3})$ and for each $t \in [0, 1]$ let $(a_m)_{m \geq 1}$ be a sequence in $\{0, 1, 2\}$ such that $t = \sum_{m \geq 1} 3^{-m} a_m$. Put 
$\kappa(t) = \lim_{m \to \infty} \kappa_{a_1} \circ \kappa_{a_2} \circ \dotsb \circ \kappa_{a_m}(z_0)$, where
\begin{align*}
\kappa_0(z) &= \beta^{-4} + \beta^{-3}z, \\
\kappa_1(z) &= \beta^{-4} + \beta^{-6} + \frac{\beta^{-10}}{1 - \beta^{-3}} - \beta^{-4}z, \\
\kappa_2(z) &= \beta^{-3} + \beta^{-4} + \beta^{-3}z.
\end{align*}
It is shown in Section 4 of \cite{rauzy} that $\kappa$ is bijective and that
$$\K' = \kappa([0, 1])=\RR \cap (\RR + \beta^{-1}).
$$
The set  $\K'$ is also part of the boundary of the tribonnacci fractal. The rest of the boundary is obtained by scaling, rotating and translating $\K'$. 

We have that $\K' \subseteq \RR_0 \cup \RR_{11}$, that is, every point in $\K'$ has a representation starting with either $0$ or $11$. Furthermore, $\K' = \kappa_0(\K') \cup \kappa_1(\K') \cup \kappa_2(\K')$, which is a consequence of Lemma 4 of \cite{rauzy}, and
$$
\K' \cap \RR_0 = \kappa_0(\K') \cup \kappa_1(\K'), \quad \K' \cap \RR_{11} = \kappa_2(\K').
$$
A simple computation shows that by applying the shift map to $\K' \cap \RR_{11}$ one obtains $\beta^{-3} + \beta^{-2} \K'$ and that by applying the shift map again one obtains $\beta^{-1} \K'$, as Figure~\ref{fig:rauzy} shows.

Finally we describe some basic additional properties of $\RR$ and $\K$ that we will need:

\begin{lem}\label{lem:basicrauzy}
One has
	\begin{enumerate}[label=(\roman{*})]
		\item $\RR \cap \left( \RR + \frac{\beta^2 - 1}{2} \right) = \varnothing$;
		\item $\K' \cap \left( \RR+ \frac{\beta^2 + 2\beta + 3}{2} \right) = \varnothing$;
		\item $\K' \cap \beta^{-2} \RR = \varnothing$.
	\end{enumerate}
\end{lem}

\subsubsection{Parametrization of the boundary of fractals associated with the cubic Arnoux-Yoccoz i.e.m.}

We will use $\K = \K' - z_0$ to parametrize the boundaries of each $\F_a$ for each $a \in \A$. It is not difficult to see (after a simple computation) that $\K$ is a curve with endpoints $\kappa(0) - z_0 = 0$ and $\kappa(1) - z_0 = (-\beta^2 - 2\beta - 1)/2$. To get our parametrization we will need the following four lemmas that at the end show that the boundaries of the $\F_a$'s are Jordan curves. 

\begin{lem} \label{lem:propertiesX}
The following equalities are satisfied:
\[
		\K  = \kappa(1) - \K  \quad \text{ and } \quad \K  = \beta^{-3} \K  \cup (\beta^{-4} \K  + \beta^{-3}) \cup (\beta^{-3}\K  + \beta^{-3}).
\]
\end{lem}

\begin{proof}
The first equality comes from the fact that $\K'$ is symmetric, as shown in Lemma 4 of \cite{rauzy}. Furthermore, since $\K' = \kappa_0(\K') \cup \kappa_1(\K') \cup \kappa_2(\K')$, 
by subtracting $z_0$ from both sides, expanding and using the first equality, we get the desired result.
\end{proof}

The following two lemmas will serve to prove that the parts of the boundaries in the Arnoux-Yoccoz fractals coming from the tribonacci fractal intersect in a unique point (see Figure \ref{fig:AY}).

\begin{lem}\label{lem:rauzy3}
	We have that $\K \cap \left( \beta \K + \frac{\beta^2 + 1}{2} \right) = \K \cap \left( \beta^2 \K + \frac{-\beta^2 + 1}{2} \right) = \{0\}$.
\end{lem}

\begin{proof}
Recall that $z_0 = \beta^{-4}/(1 - \beta^{-3}) = \frac{-\beta^2 - 2\beta - 3}{2}$ and $\K'= \RR \cap (\RR + \beta^{-1})$. Set $K_1 = \K \cap \left( \beta \K + \frac{\beta^2 + 1}{2} \right)$ and $K_2 = \K \cap \left( \beta^2 \K + \frac{-\beta^2 + 1}{2} \right)$.
	
(i) Let $z \in K_1$. We will prove that $-\beta^2 z \in K_2$. By definition, we have that $z = \zeta - z_0 = \beta (\zeta' - z_0) + \frac{\beta^2 + 1}{2}$, where $\zeta, \zeta' \in \K'$. Therefore,
$$
\zeta = \beta \zeta' - \beta z_0 + z_0 + \frac{\beta^2 + 1}{2} = \beta \zeta' + \frac{\beta^2 - 1}{2}.
$$
From the discussion in Section \ref{sec:rauzy} $\zeta'$ has a representation starting with either $0$ or $11$. 
Let $a_1, a_2, a_3$ be the first three letters of such representation and consider the point $\zeta'' = \beta( \zeta' - \beta^{-3}a_1 ) \in \RR$, that is, the point obtained by shifting the representation of $\zeta'$. We get that $\zeta = \zeta'' + \beta^{-2} a_1 + \frac{\beta^2 - 1}{2}$, where $a_1 \in \{0, 1\}$.
		
By Lemma \ref{lem:basicrauzy} item (i) we have that $\RR \cap \left(\RR + \frac{\beta^2 - 1}{2} \right) = \varnothing$, so $a_1 \not = 0$ and thus $a_1 a_2 a_3 = 110$. Then, $\zeta' \in \K' \cap \RR_{11}$, so $\zeta'' \in \beta^{-3} + \beta^{-2} \K' = \beta^{-2} \K + \frac{-\beta^2 - 2\beta - 1}{2}$ and by replacing in previous expressions the value of $a_1$ we get that $\zeta = \zeta'' + \frac{3\beta^2 + 4\beta + 3}{2}$. 

Finally, we get that $z \in \K \cap \left(\beta^{-2}\K + \frac{3 \beta^2 + 4\beta + 5}{2} \right)$, so
$$
-\beta^2 z \in \left( \frac{-\beta^2 - 2\beta - 1}{2} - \K \right) \cap (-\beta^2 \K) = K_2,
$$
		where the last equality is obtained by using Lemma \ref{lem:propertiesX} and a simple computation.
\medbreak
		
(ii) Let $z \in K_2$. We will prove that $-\beta z \in K_1$. The proof is similar to (i) so we skip some details.

By definition, we have that $z = \zeta - z_0 = \beta^2(\zeta' - z_0) + \frac{-\beta^2 + 1}{2}$, where $\zeta, \zeta' \in \K'$. Therefore,
		$$
		\zeta = \beta^2 \zeta' - \beta^2 z_0 + z_0 + \frac{-\beta^2 + 1}{2} = \beta^2 \zeta' + \frac{-\beta^2 - 2\beta - 1}{2}.
		$$
		As in the proof of (i), $\zeta'$ has a representation that begins with either $0$ or $11$. Let $a_1, a_2, a_3$ be its first three letters and $\zeta'' = \beta^2 ( \zeta' - \beta^{-3} a_1 - \beta^{-4} a_2 ) \in \RR$, that is, the point obtained by shifting the representation of $\zeta'$ twice. We have that $\zeta = \zeta'' + \beta^{-1} a_1 + \beta^{-2} a_2 + \frac{-\beta^2 - 2\beta - 1}{2}$, where $a_1, a_2 \in \{0, 1\}$.
		
		We know that $a_1 a_2 \neq 10$. If $a_1 a_2 = 01$, we would have that $\zeta = \zeta'' + \frac{\beta^2 +2\beta + 3}{2}$. By Lemma \ref{lem:basicrauzy} item (ii) we have that $\K' \cap \left( \RR + \frac{\beta^2 +2\beta + 3}{2} \right) = \varnothing$, so this cannot happen. Furthermore, by item (iii) of the same lemma we have that $\K' \cap \beta^{-2} \RR = \varnothing$, so $a_1 a_2 \neq 00$. We obtain $a_1a_2a_3 = 110$, so $\zeta = \zeta'' + \frac{3\beta^2 + 4\beta + 5}{2}$ and we deduce that $\zeta'' \in \beta^{-1} \K' = \beta^{-1} \K + \frac{-3\beta^2 - 4\beta - 5}{2}$. Finally, $z \in \K \cap \left( \beta^{-1} \K + \frac{\beta^2 + 2\beta + 3}{2} \right)$, so
		$$
			-\beta z \in \left( \frac{-\beta^2 - 2\beta - 1}{2} - \K \right) \cap (-\beta \K) = K_1,
		$$
where the last equality is obtained by using the previous lemma.
\medbreak

From (i) and (ii) we find that if $z$ belongs to either $K_1$ or $K_2$, then $\{\beta^n z,-\beta^n z\}$ meets $K_1 \cup K_2$ for infinitely many integers $n \geq 0$. Since $K_1\cup K_2$ is bounded and $|\beta|>1$, we must have that $z = 0$.
\end{proof}

\begin{lem}
\label{lem:rauzy2}
We have that $(\K-2\beta^2 - 3 \beta -4) \cap \left( \K + \frac{-3\beta^2 - 4\beta - 7}{2} \right)$ is a
unique point.  
\end{lem}

\begin{proof}
We illustrate the proof in Figure \ref{fig:rauzy2}. 
Observe that $(\K-2\beta^2 - 3 \beta -4) \subseteq \RR$. 
By the previous lemma and a translation, a rotation and a scaling, $\K-2\beta^2 - 3 \beta -4$ is a subset of $\mathop{\mathrm{int}}\RR$ up to one point. 
On the other hand, $\left( \K + \frac{-3\beta^2 - 4\beta - 7}{2} \right) \subseteq \RR - 1 - \beta^{-1}$. 
We know from \cite{rauzy} that $\RR- 1 - \beta^{-1} \subseteq \C \setminus \mathop{\mathrm{int}} \RR$, then
$\left( \K + \frac{-3\beta^2 - 4\beta - 7}{2} \right)$ is a subset of $\C \setminus \RR$ up to one point.
\end{proof}

\begin{figure}
\centering
\includegraphics[scale=0.4]{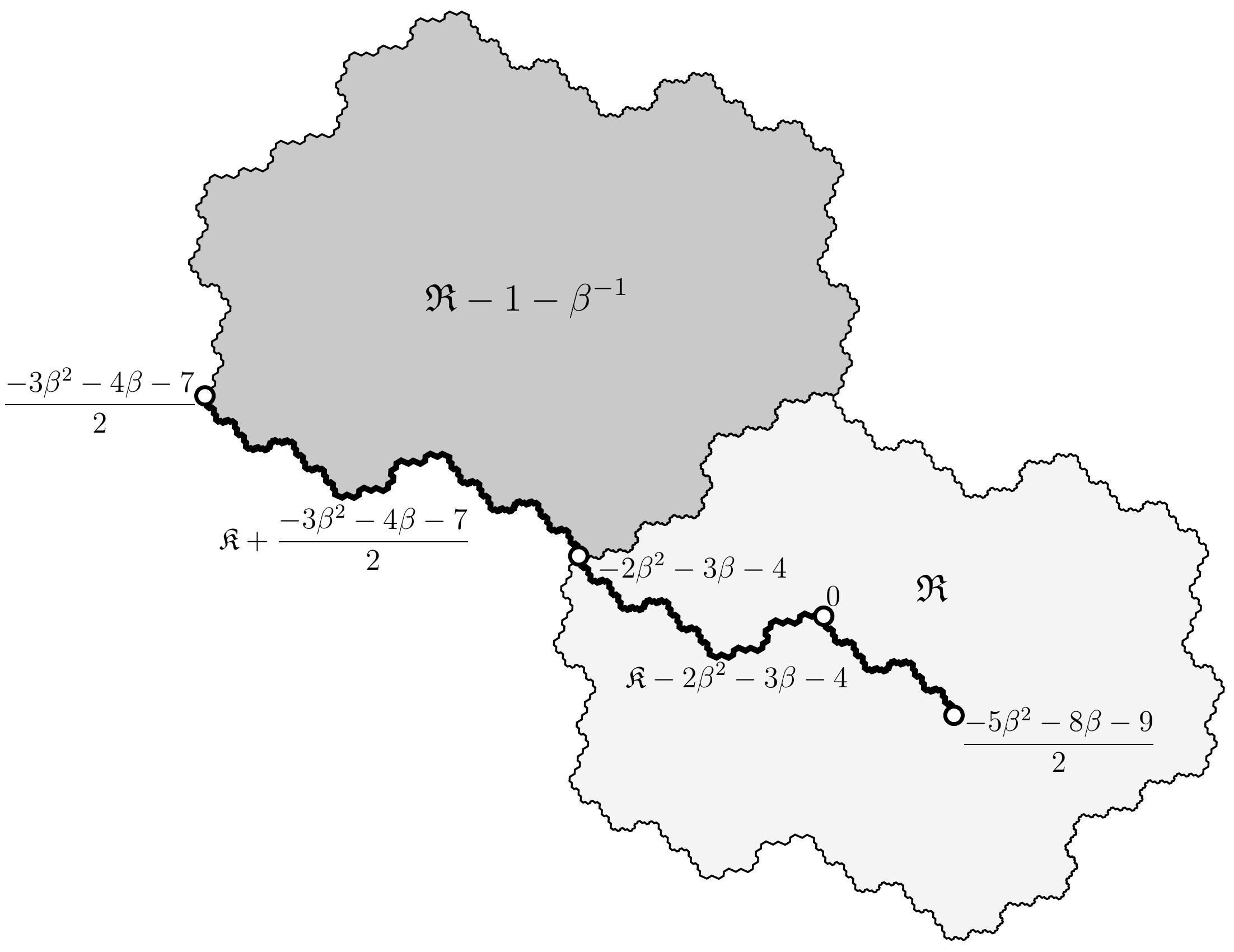}
\caption{Illustration of the proof of Lemma \ref{lem:rauzy2}.}
\label{fig:rauzy2}
\end{figure}

Now we give the parametrization of the boundaries of the Arnoux-Yoccoz fractals. Define, 
\begin{align*}
		\CC_1 &= (\beta^{-1} \K) \cup (\beta^{-2}\K + \beta^2+\beta+1) \cup \left(\K + \frac{3\beta^2 + 4\beta + 3}{2} \right) \\
		&\quad\cup \left(\beta^{-2}\K + \frac{3\beta^2 + 4\beta + 3}{2}\right) \cup \left( \beta^{-1}\K + \frac{\beta^2 + 2\beta + 1}{2} \right)  \\
		&\quad\cup \left( \K + \frac{\beta^2 + 2\beta + 1}{2} \right), \\
		\CC_2 &= \K \cup (\beta \K - \beta) \cup \left( \K + \frac{\beta^2 + 1}{2} \right) \cup \left( \beta \K + \frac{\beta^2 + 1}{2} \right), \\
		\CC_4 &= \beta \K \cup ( \beta^{-1}\K + \beta + 1 ) \cup ( \beta^{-3} \K + \beta + 1 )  \cup (\beta^{-2} \K - \beta^2 - \beta - 1) \\
		&\quad \cup \left(\beta^{-2}\K + \frac{\beta^2 + 2\beta + 3}{2} \right) \cup \left(\beta^{-1}\K + \frac{\beta^2 + 2\beta + 3}{2} \right), \\
		\CC_5 &= \beta \K \cup ( \beta^{-1}\K + \beta + 1 ) \cup \left( \beta \K + \frac{\beta^2 + 2\beta + 3}{2} \right) \cup \left(\beta^{-1}\K + \frac{\beta^2 + 2\beta + 3}{2} \right), \\
		\CC_7 &= \beta^{-1}\K \cup (\K - 1) \cup (\beta^{-2}\K - 1) \cup (\beta^{-1}\K - \beta^2 - \beta - 2) \\
		&\quad \cup \left( \beta^{-2}\K + \frac{\beta^2 + 2\beta + 1}{2} \right) \cup \left( \K + \frac{\beta^2 + 2\beta + 1}{2} \right), \\
		\CC_8 &= \beta^{-1}\K \cup (\K - 1) \cup \left( \beta^{-1} + \frac{\beta^2 + 2\beta + 1}{2} \right)  \cup \left( \K + \frac{\beta^2 + 2\beta + 1}{2} \right)
	\end{align*}
and $\CC_3 = \CC_2$, $\CC_6 = \CC_5$, $\CC_9 = \CC_8$. 

\begin{lem} For each $a \in \A$ we have that $\CC_a$ is a Jordan curve.
\end{lem}
\begin{proof}
	
In the definition of each $\CC_a$, the terms between unions correspond to the segments in the boundary of $\F_a$ shown in Figure \ref{fig:AY} in the clockwise order, starting at $0$. So it is enough to see that the intersection of two consecutive segments is a single point. Indeed, using the results for the tribonnacci fractal from \cite{rauzy} we get that the intersection of most of the pairs of contiguous segments have only one point. The conclusion of the lemma follows from the two previous lemmas.

For example, consider the segments $\beta^{-1} \K$ and $\beta^{-2} \K + \beta^2 + \beta + 1$, which are part of $\CC_1$. By translating by $-\beta^{-1} z_0$ and replacing $\K$ by $\K' - z_0$, we get that these segments in intersect in a single point if and only if $\beta^{-1} \K'$ and $\beta^{-2}\K' + \beta^{-2} + \beta^{-1}$ do. These two segments are part of the boundary the tribonacci fractal, which is a Jordan curve. Another example are the segments $\K + \frac{\beta^2 + 2\beta + 1}{2}$ and $\beta^{-1} \K$, which are also part of $\CC_1$. We can amplify by $\beta$ to get that these segments intersect in a single point if and only if $\K$ and $\beta \K + \frac{\beta^2 + 1}{2}$ do, which is implied by Lemma \ref{lem:rauzy3}.
\end{proof}


\begin{figure}[ht!]
\centering
\includegraphics[scale=0.26]{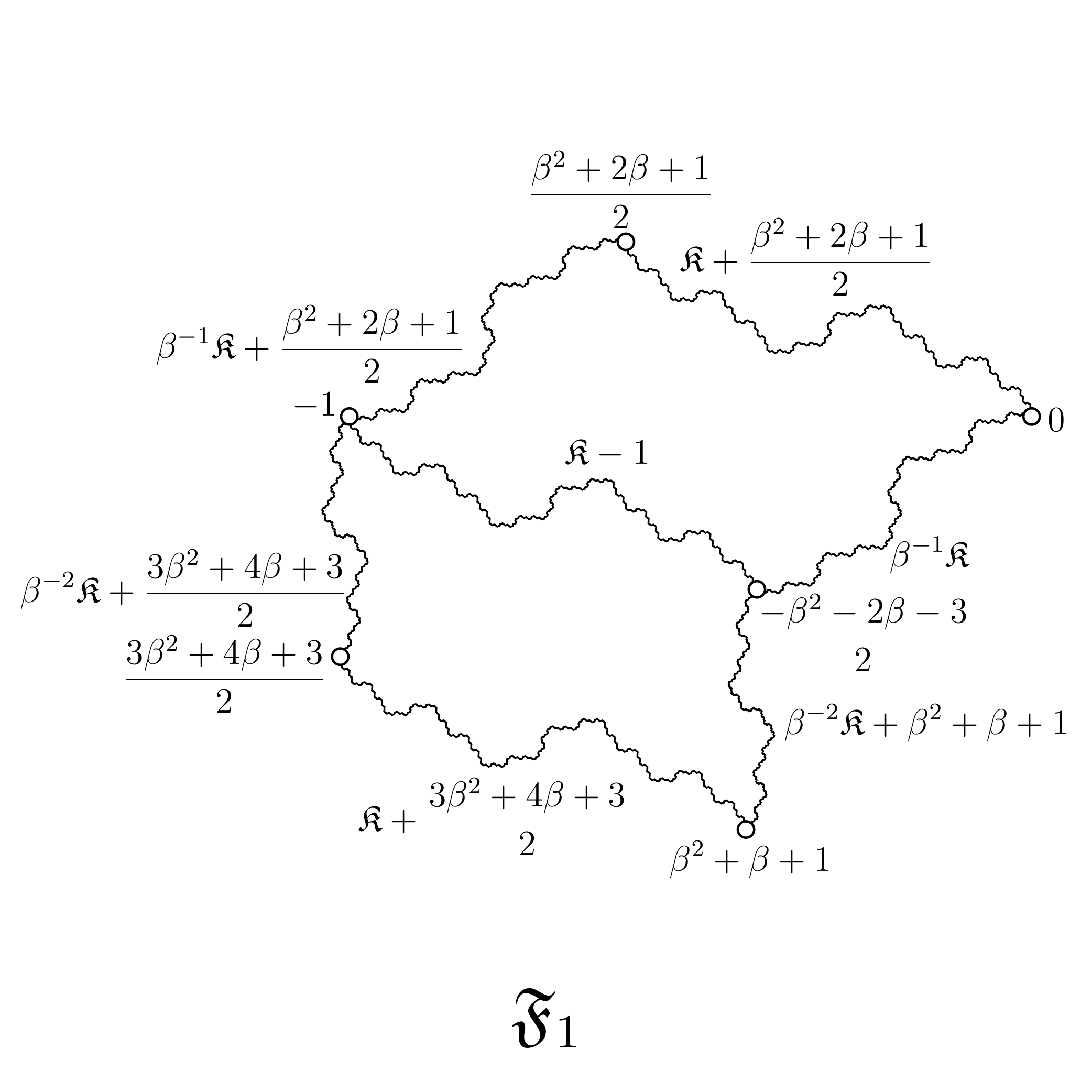} \includegraphics[scale=0.26]{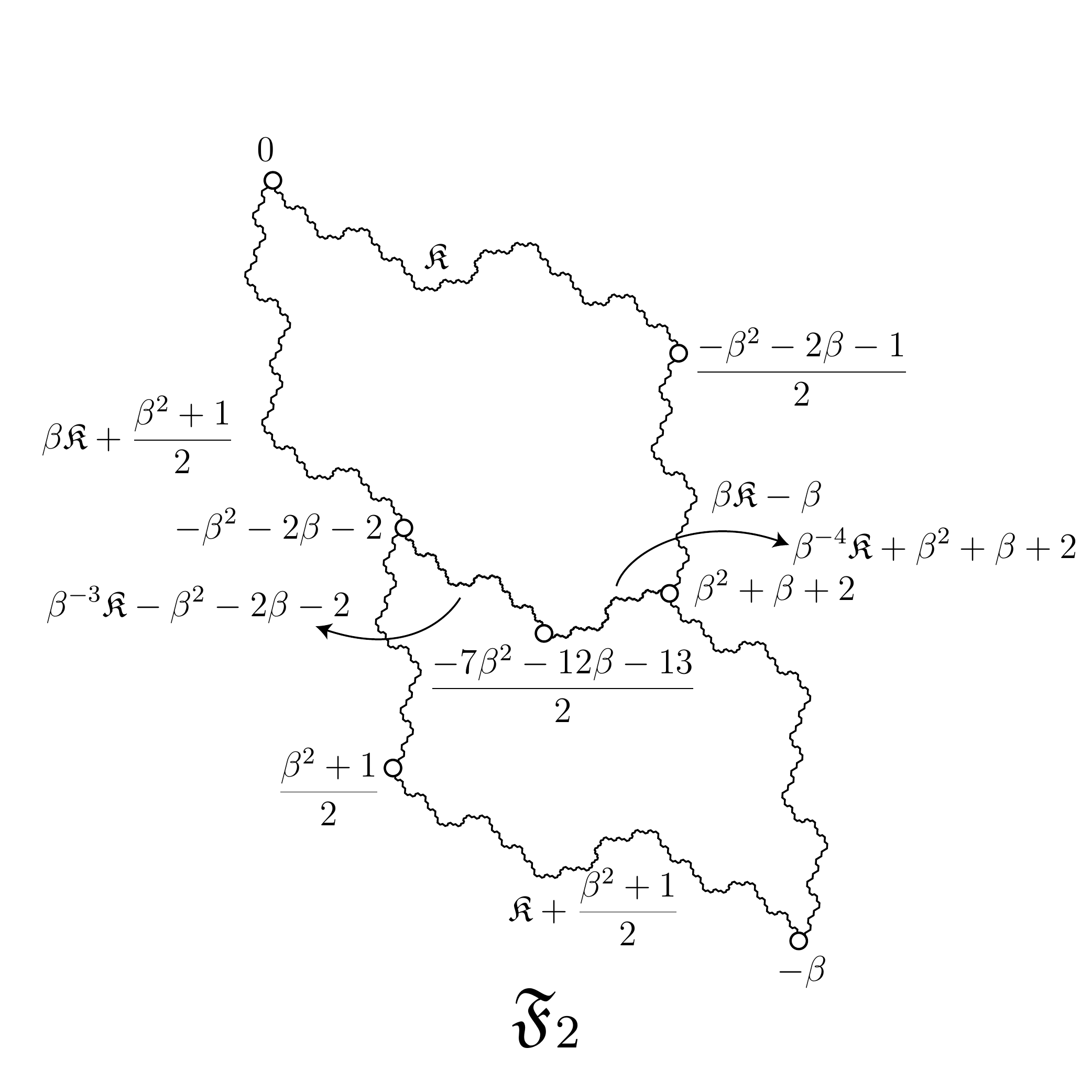}
\includegraphics[scale=0.26]{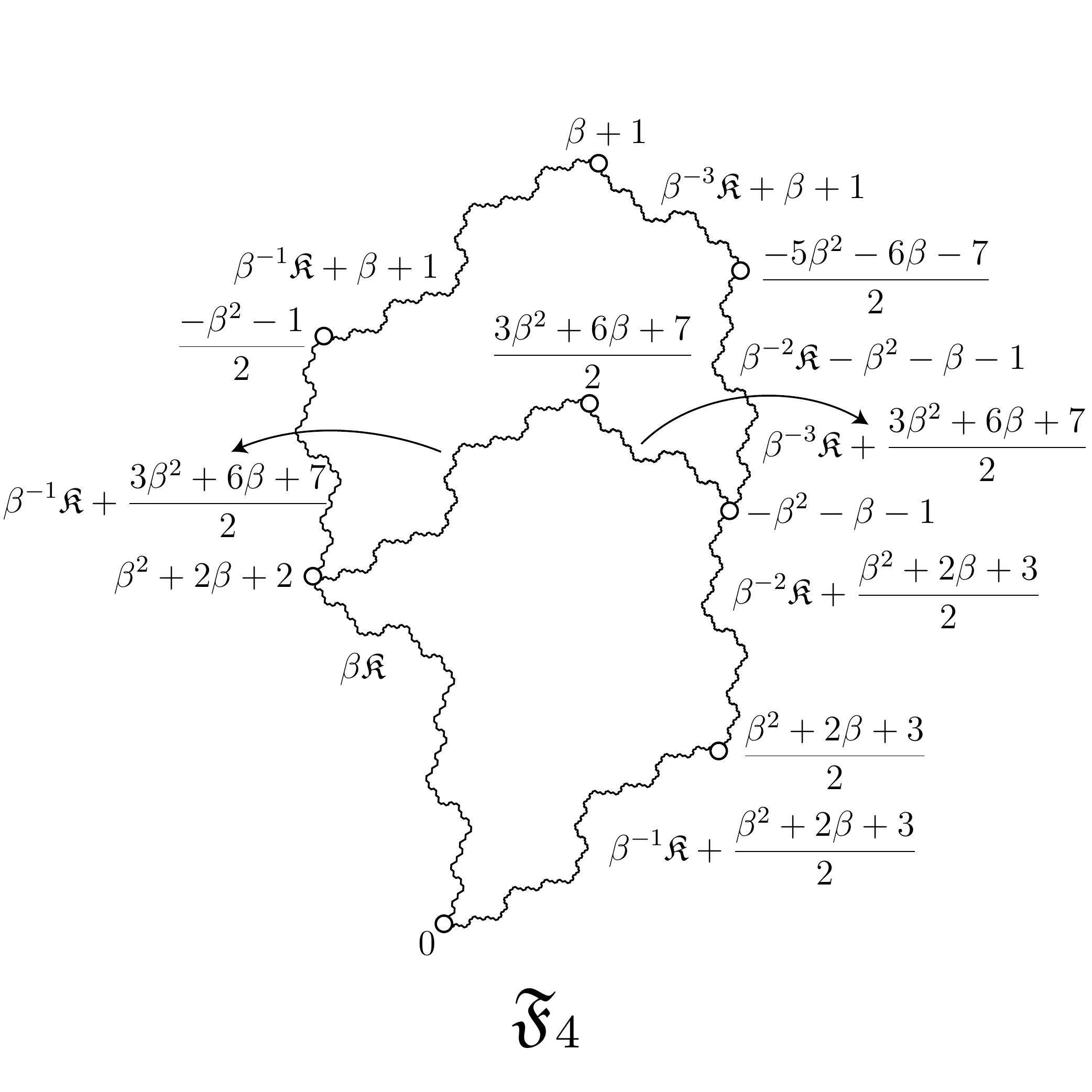} \includegraphics[scale=0.26]{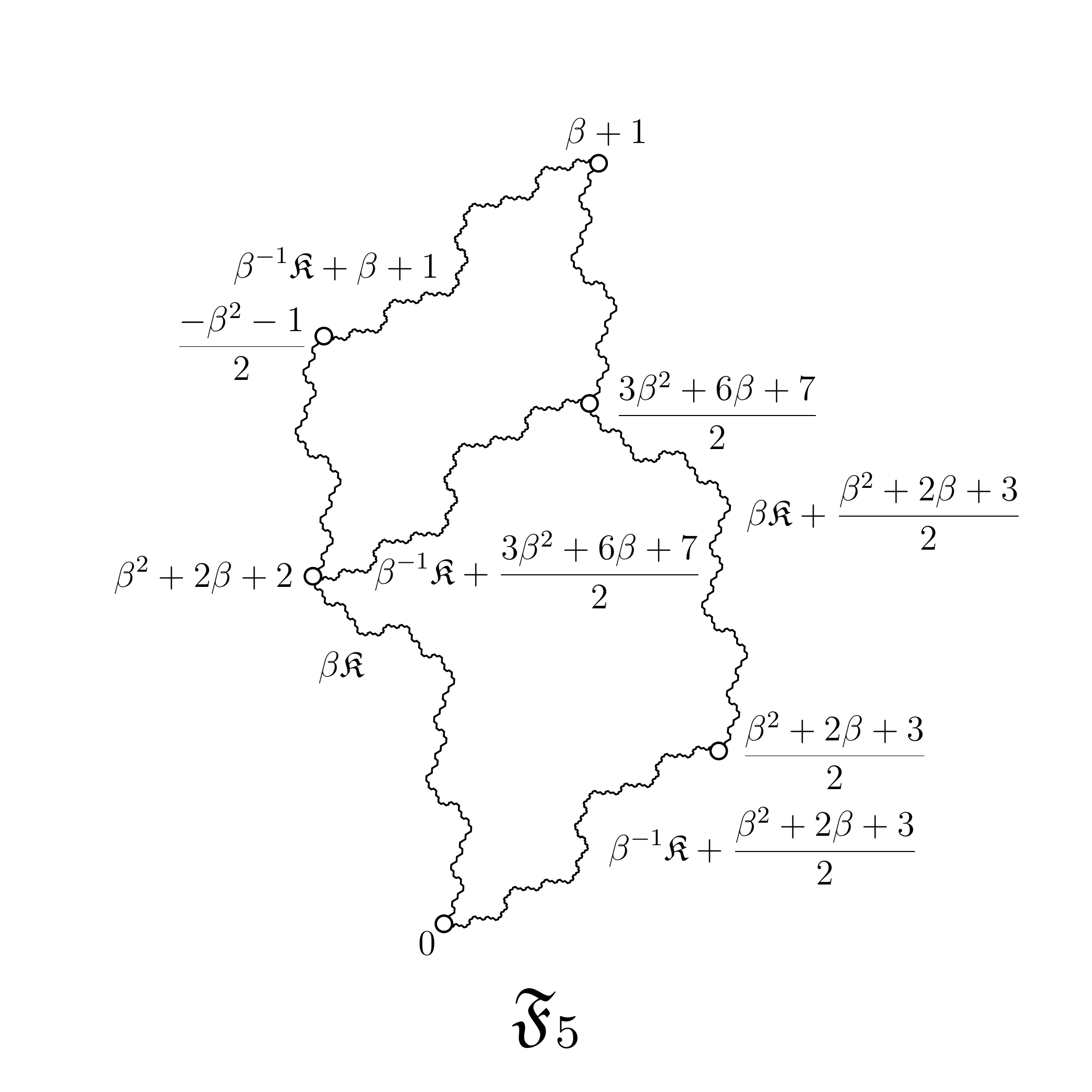}
\includegraphics[scale=0.26]{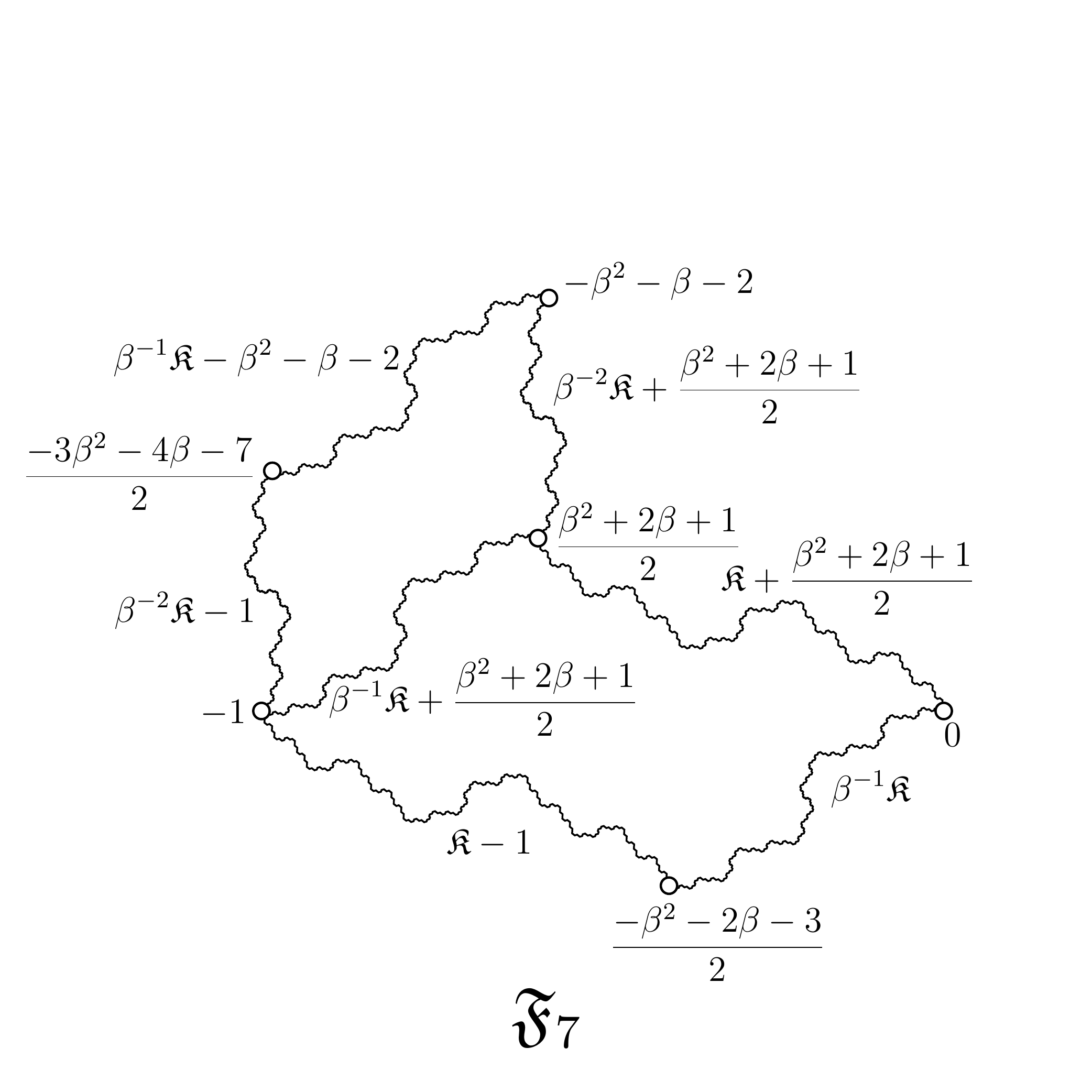} \includegraphics[scale=0.26]{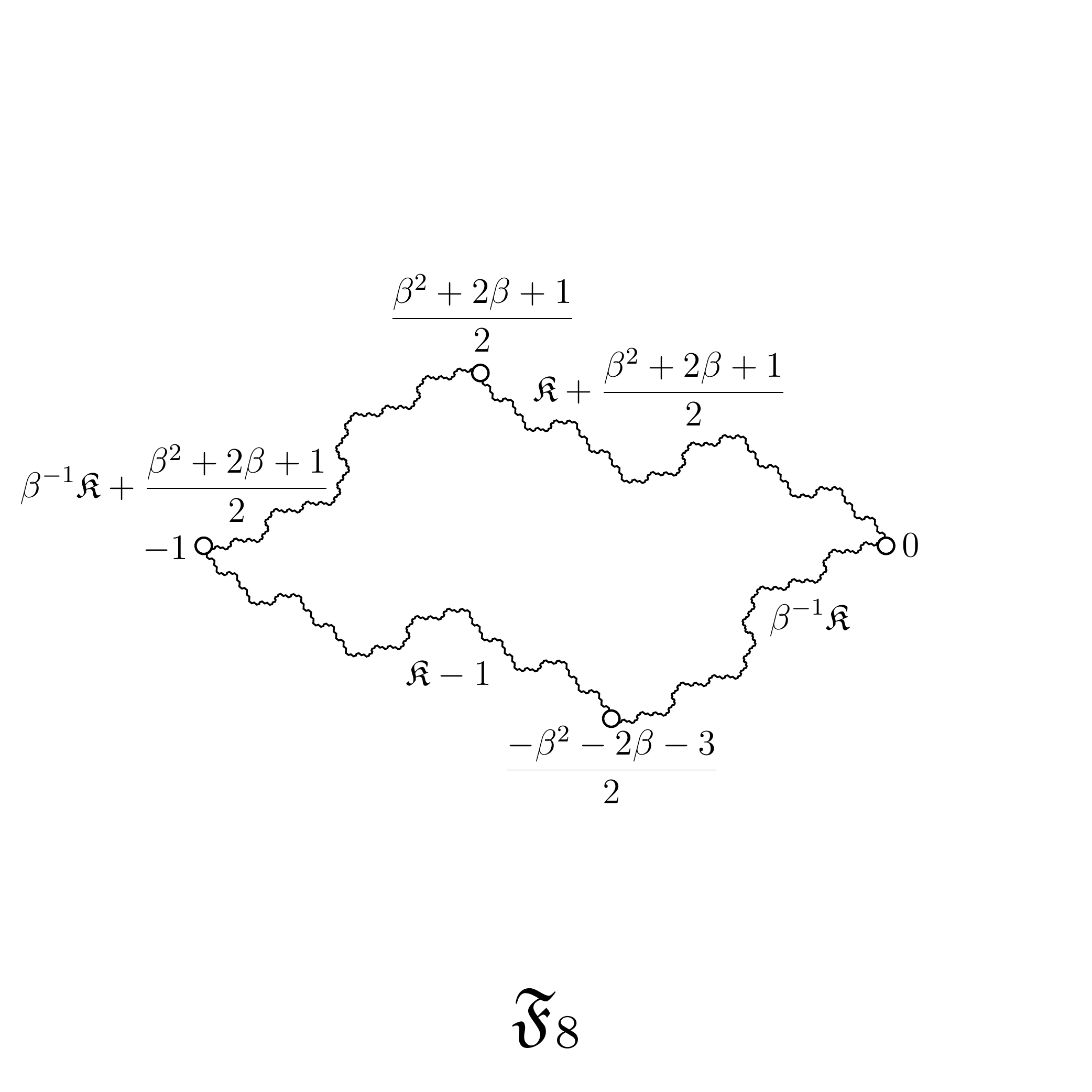}
\caption{Description of Arnoux-Yoccoz fractals and their features.}
\label{fig:AY}
\end{figure}


Now we are ready to prove that the $\CC_a$'s are parametrizations of the boundaries of the Arnoux-Yoccoz fractals. 

\begin{lem} \label{lem:disjoint-interiors}
For each $a \in \A$ we have that $\F_a$ is the closure of the interior (in the Jordan sense) of $\CC_a$. Furthermore, if $\sigma(a)$ can be written as $\sigma(a) = bc$ with $b, c \in \A$ then $\F_b \cap \mathop{\mathrm{int}} ( \gamma(b) + \F_c) = \varnothing$ and $\mathop{\mathrm{int}}  \F_b \cap ( \gamma(b) + \F_c) = \varnothing$.
\end{lem}

\begin{proof}
We will make the identifications $3 \sim 2$, $6 \sim 5$, $9 \sim 8$. By definition of $\CC_a$,
{\small\begin{align*}
		\beta^{-1} \CC_1 &= (\beta^{-2} \K) \cup (\beta^{-3}\K + \beta^2+2\beta+2) \cup \left(\beta^{-1} \K + \frac{3\beta^2 + 6\beta + 7}{2} \right)\\
		&\quad\cup \left(\beta^{-3}\K + \frac{3\beta^2 + 6\beta + 7}{2}\right) \cup \left( \beta^{-2}\K + \frac{\beta^2 + 2\beta + 3}{2} \right) \\
		&\quad\cup \left(\beta^{-1} \K + \frac{\beta^2 + 2\beta + 3}{2} \right), \\
		\beta^{-1} \CC_2 &= \beta^{-1} \K \cup (\K - 1) \cup \left(\beta^{-1} \K + \frac{\beta^2 + 2\beta + 1}{2} \right) \cup \left( \K + \frac{\beta^2 + 2\beta + 1}{2} \right), \\
		\beta^{-1} \CC_4 &= \K \cup ( \beta^{-2}\K + \beta^2 + \beta + 2 ) \cup ( \beta^{-4} \K + \beta^2 + \beta + 2 ) \\
		&\quad\cup (\beta^{-3} \K -\beta^2 - 2\beta - 2) \cup \left(\beta^{-3}\K + \frac{3\beta^2 + 4\beta + 5}{2} \right) \\
		&\quad\cup \left(\beta^{-2}\K + \frac{3\beta + 4\beta + 5}{2} \right), \\
		\beta^{-1} \CC_5 &= \K \cup ( \beta^{-2}\K + \beta^2 + \beta + 2) \cup \left( \K + \frac{3\beta^2 + 4\beta + 5}{2} \right) \\
		&\quad\cup \left(\beta^{-2}\K + \frac{3\beta^2 + 4\beta + 5}{2} \right), \\
		\beta^{-1} \CC_7 &= \beta^{-2}\K \cup (\beta^{-1} \K - \beta^2 - \beta - 1) \cup (\beta^{-3}\K - \beta^2 - \beta - 1) \\
		&\quad\cup (\beta^{-2}\K - 2\beta^2 - 3\beta - 3) \cup \left( \beta^{-3}\K + \frac{\beta^2 + 2\beta + 3}{2} \right) \\
		&\quad\cup \left(\beta^{-1} \K + \frac{\beta^2 + 2\beta + 3}{2} \right), \\
		\beta^{-1} \CC_8 &= \beta^{-2}\K \cup (\beta^{-1} \K - \beta^2 - \beta - 1) \cup \left( \beta^{-2} + \frac{\beta^2 + 2\beta + 3}{2} \right)  \\
		&\quad\cup \left(\beta^{-1} \K + \frac{\beta^2 + 2\beta + 3}{2} \right).
	\end{align*}
}

	By using Lemma \ref{lem:propertiesX} when necessary, we find that
{\small
	\begin{equation}\label{eq:boundaries}
	\begin{split}
		\beta^{-1} \CC_2 \cup \beta^{-1}(-\beta + \CC_5 ) &= \CC_1 \cup (\K - 1), \\
		\beta^{-1} \CC_4 \cup \beta^{-1} (-\beta^{-1} + \CC_5) &= \CC_2 \cup ( \beta^{-4} \K + \beta^2 + \beta + 2 ) \cup (\beta^{-3} \K -\beta^2 - 2\beta - 2), \\
		\beta^{-1} \CC_1 \cup \beta^{-1} (\beta^{-1} + \CC_7) &= \CC_4 \cup \left(\beta^{-1} \K + \frac{3\beta^2 + 6\beta + 7}{2} \right) \cup \left(\beta^{-3}\K + \frac{3\beta^2 + 6\beta + 7}{2}\right),  \\
		\beta^{-1} \CC_1 \cup \beta^{-1} (\beta^{-1} + \CC_8) &= \CC_5 \cup \left(\beta^{-1} \K + \frac{3\beta^2 + 6\beta + 7}{2} \right), \\
		\beta^{-1} \CC_2 \cup \beta^{-1} (-\beta + \CC_8) &= \CC_7 \cup \left(\beta^{-1} \K + \frac{\beta^2 + 2\beta + 1}{2} \right), \\
		\beta^{-1} \CC_2 &= \CC_8. 
	\end{split}
	\end{equation}
}

We have that the following equations stated in (6-1) of \cite{geometricalmodels} produce a unique solution for the given $\gamma$:
\begin{equation}\label{eq:aygifs}
\arraycolsep=1.5pt
	\begin{array}{llllll}
		\F_1 & = & \beta^{-1} \F_2 \cup \beta^{-1}( -\beta +  \F_5), &\qquad \F_2 & = & \beta^{-1} \F_4 \cup \beta^{-1}( -\beta^{-1} +  \F_5), \\
		\F_4 & = & \beta^{-1} \F_1 \cup \beta^{-1}( \beta^{-1} +  \F_7), &\qquad \F_5 & = & \beta^{-1} \F_1 \cup \beta^{-1}( \beta^{-1} +  \F_8), \\
		\F_7 & = & \beta^{-1} \F_2 \cup \beta^{-1}( -\beta +  \F_8), &\qquad \F_8 & = & \beta^{-1} \F_2 .
	\end{array}
\end{equation}
Finally, let $\F_a'$ be the closure of the Jordan interior of $\CC_a$. We have to show that $(\F_a')_{a \in \A}$ satisfies \eqref{eq:aygifs}. We have that (see Figure \ref{fig:AY})
\begin{align*}
&\K - 1 \subseteq \F_1',\qquad ( \beta^{-4} \K + \beta^2 + \beta + 2 ) \cup (\beta^{-3} \K -\beta^2 - 2\beta - 2) \subseteq \F_2' \\
&\left(\beta^{-1} \K + \frac{3\beta^2 + 6\beta + 7}{2} \right) \cup \left(\beta^{-3}\K + \frac{3\beta^2 + 6\beta + 7}{2}\right) \subseteq \F_4',\\
&\left(\beta^{-1} \K + \frac{3\beta^2 + 6\beta + 7}{2} \right) \subseteq \F_5', \qquad \left(\beta^{-1} \K + \frac{\beta^2 + 2\beta + 1}{2} \right) \subseteq \F_7'.
\end{align*}
So, by \eqref{eq:boundaries}, we conclude that $(\F_a')_{a \in \A}$ satisfies \eqref{eq:aygifs}, which proves that $\F_a = \F_a'$ for each $a \in \A$.

The second part of the follows directly from the parametrization (see Figure~\ref{fig:AY}).
\end{proof}
It is straightforward from the previous lemma that: 

\begin{cor}
$\partial \F_a = \CC_a$ for each $a \in \A$.
\end{cor}

\subsubsection{Unique representation property for the A-Y i.e.m.}

We finally prove that every extreme point of each fractal $\F_a$ has a unique representation.

\begin{lem}
For any $a \in \A$, extreme points in the boundary of $\F_a$ have a unique representation. That is, $T$ has the unique representation property for $\beta$.
\end{lem}
\begin{proof} 
The proof is by contradiction. We prove that extreme points in $\F_a$ with more than one representation have an eventually periodic representation. This together with Lemma \ref{lem:minimal-not-periodic} gives a contradiction. 

Let $z \in \partial\F_a$ be an extreme point with two representations. By shifting and using Lemma \ref{lem:continuation} we can assume that the first letters in these representations are different. This implies that 
$\sigma(a)$ cannot be a single letter, so we get that $\sigma(a)= bc$ for some $b, c \in \A$. Then $z \in \beta^{-1}\F_b \cap \beta^{-1}( \gamma(b) + \F_c)$ and by Lemma \ref{lem:disjoint-interiors} 
$$z \in \partial \F_a \cap (\beta^{-1}\F_b \cap \beta^{-1}( \gamma(b) + \F_c))= \partial \F_a \cap \partial\beta^{-1}\F_b \cap \partial \beta^{-1}( \gamma(b) + \F_c).$$

Now, to get the desired contradiction, we prove that these points have an eventually periodic representation. There are seven cases (see Figure \ref{fig:AY} to understand the cases): 

(i) if $a = 1$, then $z = -1$ or $z = \frac{-\beta^2 - 2\beta - 3}{2}$. The point $x_1 \in \S_a$ defined by $(p_1^{x_1}, c_1^{x_1}, s_1^{x_1}) = (3, 5, \varepsilon)$ and $p_m^{x_1} = \varepsilon$ for every $m \geq 2$ is an eventually periodic representation of $z=-1$, and the point $x_2 \in \S_a$ defined by 
$$
x_2 =(\varepsilon, 3, 5)(\varepsilon, 4, 6)(1, 7, \varepsilon)(\varepsilon, 2, 9)(\varepsilon, 4, 5)(1, 7, \varepsilon)(\varepsilon, 2, 9)(\varepsilon, 4, 5)\ldots
$$
is an eventually periodic representation of $z = \frac{-\beta^2 - 2\beta - 3}{2}$;
\medbreak
			
 (ii) if $a = 2$, then $z = -\beta^2 - 2\beta - 2$ or $z = \beta^2 + \beta + 2$. The point $x_1 \in \S_a$ defined by $(p_1^{x_1}, c_1^{x_1}, s_1^{x_1}) = (4, 5, \varepsilon)$ and $p_m^{x_1} = \varepsilon$ for every $m \geq 2$
is an eventually periodic representation of $z = -\beta^2 - 2\beta - 2$, and the point $x_2 \in \S_a$ defined by $(p_1^{x_2}, c_1^{x_2}, s_1^{x_2}) = (4, 5, \varepsilon)$, $(p_2^{x_2}, c_2^{x_2}, s_2^{x_2}) = (1, 8, \varepsilon)$ and $p_m^{x_2} = \varepsilon$ for every $m \geq 3$ is an eventually periodic representation of
$z = \beta^2 + \beta + 2$;
\medbreak

(iii) if $a = 3$, then $z = -\beta^2 - 2\beta - 2$ or $z = \beta^2 + \beta + 2$. The point $x_1 \in \S_a$ defined by $(p_1^{x_1}, c_1^{x_1}, s_1^{x_1}) = (4, 6, \varepsilon)$ and $p_m^{x_1} = \varepsilon$ for every $m \geq 2$ is an eventually periodic representation of $z = -\beta^2 - 2\beta - 2$, and the point $x_2 \in \S_a$ defined by $(p_1^{x_2}, c_1^{x_2}, s_1^{x_2}) = (4, 6, \varepsilon)$, $(p_2^{x_2}, c_2^{x_2}, s_2^{x_2}) = (1, 9, \varepsilon)$ and $p_m^{x_2} = \varepsilon$ for every $m \geq 3$ is an eventually periodic representation of $z = \beta^2 + \beta + 2$;
\medbreak

 (iv) if $a = 4$, then $z = \beta^2 + 2\beta + 2$ or $z = -\beta^2 - \beta - 1$. The point
$x_1 \in \S_a$ defined by $(p_1^{x_1}, c_1^{x_1}, s_1^{x_1}) = (1, 7, \varepsilon)$ and $p_m^{x_1} = \varepsilon$ for every $m \geq 2$ is an eventually periodic representation of $z = \beta^2 + 2\beta + 2$, and the point
$x_2 \in \S_a$ defined by $(p_1^{x_2}, c_1^{x_2}, s_1^{x_2}) = (\varepsilon, 1, 7)$, $(p_2^{x_2}, c_2^{x_2}, s_2^{x_2}) = (3, 5, \varepsilon)$ and $p_m^{x_2} = \varepsilon$ for every $m \geq 3$ is an eventually periodic representation of $z = -\beta^2 - \beta - 1$;
\medbreak

(v) if $a = 5$, then $z =\beta^2 + 2\beta + 2$ or $z = \frac{3 \beta^2 + 6\beta + 7}{2}$. The point 
$x_1 \in \S_a$ defined by $(p_1^{x_1}, c_1^{x_1}, s_1^{x_1}) = (1, 8, \varepsilon)$ and $p_m^{x_1} = \varepsilon$ for every $m \geq 2$ is an eventually periodic representation of $z = \beta^2 + 2\beta + 2$, and the point 
$x_2 \in \S_a$ defined by 
$$
x_2 =(\varepsilon, 1, 8) (3, 5, \varepsilon) (\varepsilon, 1, 8) (\varepsilon, 3, 5)(4, 6, \varepsilon)(\varepsilon, 1, 9)(\varepsilon, 3, 5)(4, 6, \varepsilon)(\varepsilon, 1, 9)\ldots
$$
is an eventually periodic representation of $z = \frac{3 \beta^2 + 6\beta + 7}{2}$;
\medbreak

 (vi) if $a = 6$, then $z =\beta^2 + 2\beta + 2$ or $z = \frac{3 \beta^2 + 6\beta + 7}{2}$. 
The point $x_1 \in \S_a$ defined by $(p_1^{x_1}, c_1^{x_1}, s_1^{x_1}) = (1, 9, \varepsilon)$ and $p_m^{x_1} = \varepsilon$ for every $m \geq 2$ is an eventually periodic representation of $z = \beta^2 + 2\beta + 2$, and 
the point $x_2 \in \S_a$ defined by 
$$
x_2=(\varepsilon, 1, 9)(3, 5, \varepsilon)(\varepsilon, 1, 8)  (\varepsilon, 3, 5)(4, 6, \varepsilon)(\varepsilon, 1, 9)(\varepsilon, 3, 5)(4, 6, \varepsilon)(\varepsilon, 1, 9)\ldots
$$
is an eventually periodic representation of $z =\frac{3 \beta^2 + 6\beta + 7}{2}$;
\medbreak

 (vii) if $a = 7$, then $z = -1$ or $z = \frac{\beta^2 + 2\beta + 1}{2}$. The point $x_1 \in \S_a$  defined by $(p_1^{x_1}, c_1^{x_1}, s_1^{x_1}) = (2, 9, \varepsilon)$ and $p_m^{x_1} = \varepsilon$ for every $m \geq 2$ is an eventually periodic representation of $z = -1$, and the point $x_2 \in \S_a$ defined by 
			$$
				x_2 =(\varepsilon, 2, 9) (4, 5, \varepsilon) (\varepsilon, 1, 8) (\varepsilon, 3, 5)(4, 6, \varepsilon)(\varepsilon, 1, 9)(\varepsilon, 3, 5)(4, 6, \varepsilon)(\varepsilon, 1, 9)\ldots
			$$
is an eventually periodic representation of $z = \frac{\beta^2 + 2\beta + 1}{2}$.
\end{proof}
\textbf{Acknowledgement.} The first author is grateful to the MathAmsud grant BEX 10674/12-8. The second and third authors are grateful to the CMM-Basal grant PFB03. We thank the anonymous referee for his or her very interesting suggestions.  

\printbibliography

\end{document}